\documentclass[reqno]{amsart}

\usepackage[alphabetic,nobysame]{amsrefs}

\usepackage[margin=1.45in,footskip=30pt]{geometry}

\usepackage[shortlabels]{enumitem}
\setlist[itemize]{nosep}
\usepackage{graphicx,pdfsync,amssymb, latexsym,amsfonts,amsbsy,color,subcaption,esint,mathtools,enumerate,nicefrac}

\usepackage{graphicx,pdfsync,amssymb, latexsym,amsfonts,amsbsy,color,subcaption,esint,mathtools,nicefrac,tikz}

\usetikzlibrary{decorations.pathreplacing}  

\usepackage{parskip}

\usepackage{hyperref}
\hypersetup{
colorlinks=true,
linkcolor=blue,
filecolor=blue,      
urlcolor=blue,
citecolor=blue,
}
\newtheorem{theorem}{Theorem}[section]
\newtheorem{lemma}[theorem]{Lemma}
\newtheorem{proposition}[theorem]{Proposition}

\newtheorem{corollary}[theorem]{Corollary}
\newtheorem{remark}[theorem]{Remark}
\newtheorem{conjecture}[theorem]{Conjecture}

\usepackage{times}
\DeclareMathOperator{\sech}{sech}
\let\div\relax
\DeclareMathOperator\div{div}

\DeclareMathOperator\supp{supp}

\title{Cascade mechanisms for Navier--Stokes blow-up}
\author [Alexey Cheskidov]{Alexey Cheskidov}

\address{Institute for Theoretical Sciences, Westlake University, China}
\email{cheskidov@westlake.edu.cn} 

\author [Mimi Dai]{Mimi Dai}

\address{Department of Mathematics, Statistics and Computer Science, University of Illinois at Chicago, Chicago, IL 60607, USA}
\email{mdai@uic.edu} 

\author [Stan Palasek]{Stan Palasek}

\address{Department of Mathematics, Princeton University, Princeton, NJ 08540, USA}
\email{spalasek@princeton.edu} 

\begin{document}

\begin{abstract}

In a recent preprint (arXiv:2511.09556), we exhibited an inverse energy cascade for the 3D Navier--Stokes equations which results in ``instantaneous'' Type I blow-up and failure of uniqueness in certain sharp regularity classes. Our purpose here is to elucidate this phenomenon further in the setting of the Obukhov dyadic model, and to compare it to the better-known phenomenon of finite-time blow-up. When intermittency is low ($\alpha\leq 2$), we recover our previous result from the Navier--Stokes setting; in the energy-supercritical case where intermittency is high ($\alpha>2$), we show that the same can occur in the class of finite energy solutions. We present two different proofs: a soft approach using a Lyapunov function and an explicit multiscale construction. For the inviscid system we illustrate that a similar phenomenon occurs at lower regularity. Finally, we state a finite-time blow-up theorem for a mixed Desnyansky--Novikov--Obukhov model, in which a forward cascade from finitely supported data and no force produces a singularity that is of Type~II only by a small margin, with features comparable to the recent forced Navier--Stokes blow-up of OpenAI.

\end{abstract}

\maketitle

\begingroup
\setlength{\parskip}{0pt}
\tableofcontents
\endgroup

\section{Introduction}

Consider the incompressible Navier--Stokes equations on $\mathbb T^d$ for $d\geq2$, given by
\begin{equation}\label{eq:NSE}
\begin{cases}
\partial_t u -  \nu \Delta u + \div( u \otimes u) + \nabla p = 0 &\\
\div u = 0.
\end{cases}
\end{equation}
When equipped with smooth initial data $u(\cdot,0)=u_0\in C^\infty$, the system \eqref{eq:NSE} is locally well-posed, meaning there exists a unique smooth solution on the time interval $[0,T)$ for some $T>0$. A significant open question when $d\geq3$ is whether the maximum lifespan $T$ can be finite, a scenario known as \emph{finite-time blow-up} which is characterized by certain norms of the solution becoming infinite. A very recent paper \cite{openai2026} by OpenAI claims the existence of a finite-time blow-up for the problem with a smooth force. In \cite{cheskidov2025instantaneous}, the authors constructed solutions with smooth initial data that exhibit \emph{instantaneous blow-up}; that is,
\begin{itemize}[left=1em]
\item $u$ is a classical solution of \eqref{eq:NSE} on $\mathbb T^d\times([0,T]\setminus \{T_*\})$,
\item $u(\cdot,T_*)$ is spatially smooth and remains a genuine weak solution at $t=T_*$, and
\item the solution blows up from the right in the sense that
\begin{align*}
    \lim_{t \to T_*+} \|u(t)\|_{L^\infty}=\infty.
\end{align*}
\end{itemize}

The main purpose of this work is to study, in the simplified setting of dyadic models, the energy cascades behind both kinds of singularity: the inverse cascade that produces instantaneous blow-up, which occupies most of the paper, and the forward cascade that produces finite-time blow-up, which we take up in Section~\ref{sec:mixed}. We begin with the viscous Obukhov model, which can be expressed as
\begin{equation} \label{Obukhov}
X_k'= - \nu N_k^2X_k + N_k^\alpha X_{k+1}^2 - N_{k-1}^\alpha X_{k-1} X_k, \qquad k=1,2, \dots; \qquad X_0=0
\end{equation}
which is equivalent to the classical Obukhov model, with $\{X_k\}_{k\geq 1}$ replaced by $\{-X_k\}_{k\geq 1}$ for convenience. Throughout, $\nu>0$ is fixed. When $N_k=\lambda^k$, we use the convention $N_0=1$; its value is immaterial in \eqref{Obukhov} because $X_0=0$. The parameter $\alpha$ encodes the intermittency of the flow being modeled.
Suppose that the part of a velocity field at frequency $N$ is concentrated
on a set of dimension $D_{in}\in[0,d]$, so that it occupies a volume
fraction $N^{-(d-D_{in})}$ of $\mathbb T^d$.  If $X_N:=\|P_Nu\|_{L^2}$, the
amplitude of $P_Nu$ on its support is $\|P_Nu\|_{L^\infty}\sim N^{(d-D_{in})/2}X_N$,
and the rate $N\|P_Nu\|_{L^\infty}$ at which the nonlinearity acts becomes
$N^{\alpha}X_N$ with
\begin{equation}\label{eq:alpha}
\alpha=\frac{d+2-D_{in}}{2}.
\end{equation}
Thus $\alpha$ ranges over $[1,(d+2)/2]$, with $\alpha=1$ corresponding to no
intermittency ($D_{in}=d$) and $\alpha=(d+2)/2$ to the most intermittent
scenario ($D_{in}=0$); larger values of $\alpha$ correspond to more
intermittency and, potentially, to more pathological behavior.  The
dividing line $\alpha=2$ of our results corresponds to $D_{in}=d-2$. The
energy-supercritical range $\alpha>2$ requires $D_{in}<d-2$, and is therefore
empty when $d=2$, in agreement with the well-posedness of the
two-dimensional equations in $L^2$.  Under this correspondence, $\|u\|_{L^\infty}$ scales like $\|X\|_{H^{\alpha-1}}$ and the critical Navier--Stokes space
$B^{-1}_{\infty,\infty}$ scales like the weighted supremum
$\|X\|_{B_{2,\infty}^{\alpha-2}}\coloneqq\sup_kN_k^{\alpha-2}|X_k|$. We use the dyadic Sobolev norm
\[
\|X\|_{H^s}^2
:=
\sum_{k\geq1}N_k^{2s}X_k^2,
\qquad s\in\mathbb R
\]
for which the scaling-critical Sobolev exponent of
\eqref{Obukhov} is $s_c=\alpha-2$. Finally, we write $H^0=\ell^2$.

Models such as the Obukhov model \eqref{Obukhov} are known to share certain common behaviors with the fluid PDEs that they model and, in certain special cases, ideas can be transferred from the dyadic setting to the full PDE setting. For example, non-uniqueness of Leray--Hopf solutions of \eqref{Obukhov} from critical initial data was established by the third author~\cite{palasek2024non}, which motivated the non-uniqueness result for the Navier--Stokes equations in~\cite{MR5008166}. Global regularity of the inviscid Obukhov model from data in $H^s$ with
$s>\alpha$ was proved by Kiselev--Zlato\v s~\cite{MR2180809},
who left open the endpoint $s=\alpha$ and suggested that the result should
persist in the presence of viscosity.  For the viscous model the picture
depends on $\alpha$: when $\alpha<2$ the critical exponent $s_c=\alpha-2$ is
negative and the model is globally well-posed in $\ell^2$, so that the non-uniqueness of
Theorem~\ref{th:alpha<2} necessarily involves infinite energy; when
$\alpha>2$ the energy space is supercritical and finite-time blow-up is
expected, and was established with a smooth force and super-exponentially
growing frequencies in~\cite{palasek2026finite}. Theorem~\ref{th:DNO} below obtains finite-time blow-up with geometric frequencies and no force, once a small Desnyansky--Novikov component is added to the nonlinearity.

A close cousin of \eqref{Obukhov} is the Desnyansky--Novikov (DN) dyadic model, whose modern study was initiated by Katz--Pavlovi\'c~\cite{MR2095627} and Friedlander--Pavlovi\'c~\cite{MR2038114}. Here, a great deal is known about finite-time blow-up and forward energy cascade for both the inviscid and viscous systems, with and without a force; see for instance~\cite{MR2038114,MR2095627,MR2180809,MR2337019,MR2415066,MR2600714,MR2844828,MR3415578}. While the DN model has been the subject of much more attention than Obukhov, it appears that none of the negative results have been successfully transferred to the PDE setting.\footnote{On the other hand, certain \emph{positive} results have been motivated by dyadic models; see, for instance, \cite{MR1911664} and \cite{MR3318746}.} For a more complete review of the theory of the DN and other models, we refer the reader to the survey article~\cite{MR4607726}.

\subsection{Inverse energy cascade and instantaneous blow-up}\label{sec:backward}

Our results for the viscous version of the Obukhov model fall into three categories depending on the degree of intermittency. First, in the energy-subcritical case, we have the following. 

\begin{theorem}[Energy-subcritical instantaneous blow-up]\label{th:alpha<2}
Let $\alpha\in[1,2)$ and $N_k = \lambda^{k}$ for some $\lambda>2^{1/(2-\alpha)}$. Then there exists $X\in C([0,\infty);H^s)$ for all $s<\alpha-2$, satisfying \eqref{Obukhov}, with
\[
\lim_{t\to0+}X_k(t) = 0, \qquad k=1,2,\dots,
\]
\[X\in C((0,\infty);H^s) \qquad \forall s\in\mathbb R,\]
and
\begin{align}\label{blow-up_rate}
\limsup_{t \to 0+} t^{\frac{r-\alpha+2}2}\|X(t)\|_{H^r} >0
\end{align}
for all $r>\alpha-2$.

In particular, $X$ and the zero solution are distinct solutions of
\eqref{Obukhov} with the same datum, in $C([0,\infty);H^s)$ for every
$s<\alpha-2$ and in $L^\infty_tB^{\alpha-2}_{2,\infty}$.
\end{theorem}

\begin{remark}
Our prior result \cite{cheskidov2025instantaneous} in the full Navier--Stokes setting corresponds to Theorem~\ref{th:alpha<2} upon taking\footnote{More precisely, the construction in \cite{cheskidov2025instantaneous} is logarithmically more intermittent than $\alpha=1$.} $\alpha=1$. Indeed, $\|u(t)\|_{L^\infty}$ corresponds to $\|X(t)\|_{H^{\alpha-1}}$ so the blow-up rate \eqref{blow-up_rate} aligns with the result
\[
\limsup_{t\to0+}t^\frac12\|u(t)\|_{L^\infty}\in(0,\infty)
\]
obtained there.
\end{remark}

In the case $\alpha=2$, the nonlinear terms and the linear term in \eqref{Obukhov} appear to be balanced, and the system is considered energy-critical. In this case, we have the following.
\begin{theorem}[Energy-critical instantaneous blow-up]
\label{th:critical-blowup}
Let $\alpha=2$, $N_k=\lambda^k$ for some $\lambda>1$, and let $\sigma>0$
satisfy
\begin{equation}\label{eq:sigma-separation}
\lambda^\sigma\geq2.
\end{equation}
Then there exists a non-negative, nontrivial solution $X$ of
\eqref{Obukhov} such that
\[
X\in\bigcap_{s<-\sigma}C([0,\infty);H^s)
\cap\bigcap_{r\in\mathbb R}C((0,\infty);H^r),
\qquad X(0)=0,
\]
and
\begin{equation}\label{eq:critical-energy-blowup}
\lim_{t\to0+}\|X(t)\|_{\ell^2}=\infty.
\end{equation}
More precisely, there exist constants $c,t_0>0$ such that
\begin{equation}\label{eq:log-lower-bound}
\|X(t)\|_{\ell^2}
\geq c\left(1+\log\frac{t_0}{t}\right)^{1/3},
\qquad 0<t<t_0.
\end{equation}
\end{theorem}

Next we consider the energy-supercritical Obukhov model. This corresponds to solutions of the 3D Navier--Stokes equations with sufficient intermittency to (potentially) allow finite-time blow-up, non-uniqueness of Leray--Hopf solutions, etc.

\begin{theorem}[Energy-supercritical instantaneous blow-up]\label{th:alpha>2}
Let $\alpha>2$ and $N_k = \lambda^{k}$ for $k\geq0$, for some $\lambda>1$. Then there exists $X\in C((0,\infty);H^s)$ for all $s\in\mathbb R$, satisfying \eqref{Obukhov}, with
\[
\lim_{t\to0+}X_k(t) = 0, \qquad k=1,2,\dots,
\]
and
\[
\lim_{t \to 0+} \|X(t)\|_{\ell^2} =1.
\]
Further, every positive Sobolev norm blows up at the rate dictated by the
scaling: for $0<s<\alpha$ and $\varepsilon>0$,
\[
t^{-s/\alpha}\lesssim_s\|X(t)\|_{H^s}\lesssim_{s,\varepsilon}\,t^{-s/\alpha-\varepsilon},
\qquad 0<t\leq t_s.
\]

\end{theorem}

Based on Theorem~\ref{th:alpha>2}, we conjecture the following for the true Navier--Stokes equations.

\begin{conjecture}\label{conjecture}
There exists a weak solution $u\in L_t^\infty L_x^2\cap L_t^2\dot H_x^1([0,\infty)\times\mathbb T^3)$ of the 3D Navier--Stokes equations with instantaneous blow-up at $t=0$ obeying both
\begin{align*}
u(t)\rightharpoonup0\quad\text{weakly in }L^2(\mathbb T^3)\text{ as }t\to0
\end{align*}
and
\begin{align*}
    \lim_{t \to 0+} \|u(t)\|_{L^2} =1.
\end{align*}

\end{conjecture}

Such a $u$ would be a second finite-energy solution with zero datum, lying
in the regularity class of Leray--Hopf solutions but violating the energy
inequality at the initial time.

\begin{remark}
We remark that Conjecture~\ref{conjecture} is quite distinct from the well-known problem of uniqueness of Leray--Hopf solutions, which may exhibit energy jumps but only in the direction of \emph{losing} energy relative to the data. One distinction is that Type~I forward self-similar constructions as explored in \cite{MR3341963,MR4593274,hou2025nonuniqueness,ionescu2026non} are ruled out as candidates for the conjecture. Indeed, a self-similar solution would lie in the critical space $L^\infty_tB^{-1}_{\infty,\infty}$, and interpolation with the energy space $u\in L^2_t\dot H^1$ places such a solution in
$L^3_tB^{1/3}_{3,c_0}$, a class in which the energy equality holds
\cite{MR2422377}.  A solution of Conjecture~\ref{conjecture} must therefore
have an unbounded critical norm as $t\to0+$, i.e.\ it must be of Type~II (see Figure~\ref{fig:type1-type2}).
We confirm this condition in Section~\ref{sec:second_approach} which predicts that the component of the
solution localized near frequency $N$ obeys $\|P_Nu\|_{L^\infty}\sim N^{\alpha-1}$
near the time scale $N^{-\alpha}$, hence
$\|u(t)\|_{B^{-1}_{2,\infty}}\sim t^{-(\alpha-2)/\alpha}$ with
$\alpha>2$. This should be
contrasted with the case $\alpha<2$ of Theorem~\ref{th:alpha<2}, where the
solution is bounded in the critical space and the blow-up is of Type~I.
\end{remark}

\begin{figure}[ht]
\centering
\includegraphics[width=.6\textwidth]{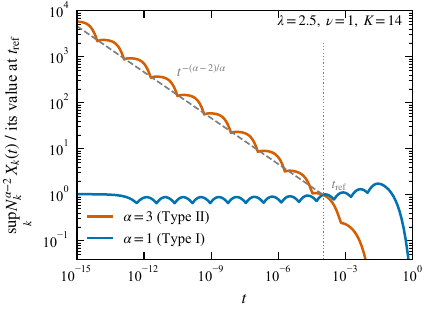}
\caption{The dyadic critical norm $\sup_kN_k^{\alpha-2}X_k$ (which scales, for instance, like
$\|u\|_{B^{-1}_{2,\infty}}$) for the supercritical
solution of Theorem~\ref{th:alpha>2} ($\alpha=3$) and the subcritical solution of
Theorem~\ref{th:alpha<2} ($\alpha=1$). As $t\to0+$ the
supercritical critical norm grows at the predicted rate $t^{-(\alpha-2)/\alpha}$, so the blow-up is of Type~II, while the subcritical one fails to grow, i.e.,\ Type~I.}
\label{fig:type1-type2}
\end{figure}

Next, in the case of the inviscid Obukhov model, we have the following results.

\begin{theorem}[Inviscid instantaneous blow-up]
\label{th:inviscid}
Let $\alpha>0$ and $N_k=\lambda^k$, where $\lambda>1$.  There exists a
non-negative solution $X$ of \eqref{Obukhov-inviscid}, componentwise classical
on $(0,\infty)$, such that, after setting $X(0)=0$,
\[
X\in C([0,\infty);H^s)
\quad\text{for every }s<0,
\qquad
X\in C((0,\infty);H^s)
\quad\text{for every }s<\alpha,
\]
and
\begin{equation}\label{inviscid-energy-jump}
\lim_{t\to0+}\|X(t)\|_{H^s}=0
\quad\text{for every }s<0,
\qquad
\|X(t)\|_{\ell^2}=1\quad\text{for every }t>0.
\end{equation}
Moreover, for every $0<r<\alpha$, there exist constants $c_r,t_r>0$ such
that
\begin{equation}\label{inviscid-blowup-rate}
\|X(t)\|_{H^r}\geq c_rt^{-r/\alpha},
\qquad 0<t\leq t_r.
\end{equation}
Consequently, for every $s<0$, the zero solution and $X$ are distinct
componentwise solutions in $C([0,\infty);H^s)$ with the same initial value.
\end{theorem}

Since the inviscid Obukhov model is time reversible, we also have

\begin{corollary}[Inviscid finite-time blow-up]
\label{cor:inviscid-forward-blowup}
Let $\alpha>0$ and $N_k=\lambda^k$, where $\lambda>1$.  Given any $T>0$,
there exists a non-positive solution $Y$ of \eqref{Obukhov-inviscid}, with
$Y_k\in C^1([0,\infty))$ for every $k\geq1$, such that
\[
Y(t)=0\quad\text{for }t\geq T,
\qquad
\|Y(t)\|_{\ell^2}=1\quad\text{for }0\leq t<T.
\]
Moreover,
\begin{equation}\label{inviscid-forward-regularity}
Y\in\bigcap_{s<\alpha}C([0,T);H^s)
\cap\bigcap_{s<0}C([0,\infty);H^s),
\end{equation}
and for every $0<r<\alpha$, there are constants $c_r>0$ and
$0<\delta_r\leq T$ such that
\begin{equation}\label{inviscid-forward-blowup-rate}
\|Y(t)\|_{H^r}
\geq c_r(T-t)^{-r/\alpha},
\qquad T-\delta_r<t<T.
\end{equation}
\end{corollary}

\begin{remark}
\begin{itemize}
\item [(i)] The solution in Corollary~\ref{cor:inviscid-forward-blowup} reverses the inverse cascade of Theorem~\ref{th:inviscid}: as $t\to T-$, its energy travels toward arbitrarily high shells.  Thus the positive Sobolev blow-up is caused by a forward cascade to infinity, not by growth of the total energy.  The energy remains equal to one before $T$ and then drops to zero under the terminal continuation, giving a complete anomalous energy loss (See Figure~\ref{fig:inviscid}).
\item [(ii)] We note that $Y(0)$ belongs to every $H^r$ with $r<\alpha$. We do not expect that it is possible to establish a finite-time blow-up for \eqref{Obukhov-inviscid} starting from initial data in $H^r$ with $r>\alpha$. In fact, in the case\footnote{For the inviscid model there is no generality lost by fixing $\alpha=1$, since the parameter $\lambda$ can be modified freely.} $\alpha=1$, Kiselev and Zlato\v s showed in \cite{MR2180809} that there exists a unique global solution $X\in C([0,T]; H^r)$ to \eqref{Obukhov-inviscid} for any $T>0$, starting from initial data in $H^r$ with $r>1$. In this sense, the blow-up result in Corollary~\ref{cor:inviscid-forward-blowup} is sharp.
\item [(iii)] Corollary~\ref{cor:inviscid-forward-blowup} may also be compared to the recent blow-up of \eqref{Obukhov} constructed by the third author in~\cite{palasek2026finite}. In that case, the initial data is smooth (i.e., $H^s$ for all $s\in\mathbb R$), at the cost of replacing the frequencies $N_k=\lambda^k$ with a super-exponential sequence. 
\item [(iv)] Whether blow-up can occur from data in the marginal space $H^\alpha$, originally asked in \cite{MR2180809}, remains open. In the larger critical space $B^\alpha_{2,\infty}$ with norm $\sup_k N_k^\alpha|X_k|$, blow-up is elementary. Indeed, $Z_k:=N_k^\alpha X_k$ solves
\begin{align*}
Z_k'=\lambda^{-2\alpha}Z_{k+1}^2-Z_{k-1}Z_k,
\end{align*}
whose right-hand side is locally
Lipschitz on $\ell^\infty$; hence
$L(t):=\lim_{k\to\infty}Z_k(t)$ satisfies $L'=-(1-\lambda^{-2\alpha})L^2$. If
$L(0)=-c<0$ (for instance $X_k(0)=-cN_k^{-\alpha}$, which lies in $H^r$ for every
$r<\alpha$), then $\|X(t)\|_{B^\alpha_{2,\infty}}\geq|L(t)|\to\infty$ no later than
$t=\big((1-\lambda^{-2\alpha})c\big)^{-1}$. Note that this mechanism is unavailable for data
in, for instance, $H^\alpha$, since $L\equiv0$.
\end{itemize}
\end{remark}

\begin{figure}[ht]
\centering
\includegraphics{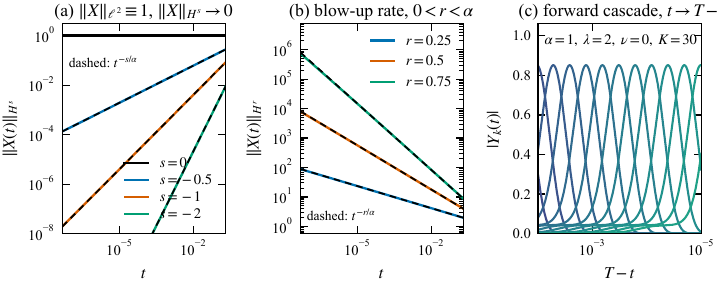}
\caption{The inviscid model \eqref{Obukhov-inviscid} with $\alpha=1$, $\lambda=2$.
(a) The energy is conserved, while
$\|X(t)\|_{H^s}\to0$ as $t\to0+$ for every $s<0$, demonstrating the energy jump of Theorem~\ref{th:inviscid}, and with it the
failure of uniqueness against the zero solution.  (b)~$\|X(t)\|_{H^r}$ for
$0<r<\alpha$, against the blow-up rate \eqref{inviscid-blowup-rate} (dashed).
(c)~The reversed solution $Y(t)=-X(T-t)$ of
Corollary~\ref{cor:inviscid-forward-blowup}: as $t\to T-$ the energy travels
through arbitrarily high shells.}
\label{fig:inviscid}
\end{figure}

In the dyadic setting (i.e., for solutions of \eqref{Obukhov}), we obtain the claimed examples of instantaneous blow-up via a Lyapunov function argument. These arguments have precedent in work by the first author on finite-time blow-up for the viscous DN model~\cite{MR2415066}. The details are supplied in Section~\ref{sec:first_approach}. Clearly, the true Navier--Stokes equations do not admit Lyapunov functions of this type, so this method cannot be expected to extend to, say, a resolution of Conjecture~\ref{conjecture}. To assuage this concern, we outline in Section~\ref{sec:second_approach} an alternate construction that does not face obstructions of this type. This second approach appears more robust and has been successfully applied in the PDE setting \cite{MR5008166,cheskidov2025instantaneous}, for instance.

\subsection{Mixed cascade and finite-time blow-up}\label{sec:mixed}

Consider the mixed Desnyansky--Novikov--Obukhov (DN-O) model
\begin{equation} \label{DNO}
X_k'= - \nu N_k^2X_k+\eta(N_{k-1}^\alpha X_{k-1}^2-N_k^\alpha X_kX_{k+1}) - N_k^\alpha X_{k+1}^2 + N_{k-1}^\alpha X_{k-1} X_k,
\end{equation}
for $k=1,2, \dots$, and $X_0=0$. We take $\nu=1$, $\alpha \geq 1$,  and $\eta=10^{-5}$. The frequency is $N_k=2^k$. Note that the DN terms have a small factor $\eta$ in front. We also point out that the sign chosen for Obukhov part in \eqref{DNO} is consistent with the conventional Obukhov model, but different from \eqref{Obukhov} above. 

\begin{theorem}[Finite-time blow-up]
\label{th:DNO}
There exists $\varepsilon >0$ such that for any $\alpha \in(\frac52 - \varepsilon, \frac52 + \varepsilon)$, there exists initial data $X(0)\in \ell^2$ supported on finitely many modes such that the DN-O model \eqref{DNO} has 
a smooth solution $X(t)$ on $[0,T)$ with $T<\infty$ satisfying
\[
 \lim_{t\to T}\|X(t)\|_{H^s}=\infty,
 \qquad \mbox{for any} \quad s> s_0
\]
for an $s_0\in (0,1/2)$.
\end{theorem}

\begin{remark}
    The intermittency parameter $\alpha=5/2$ is a key threshold for viscous dyadic models of the Navier--Stokes equations in three dimensions specifically. This value corresponds to the maximal intermittency for a 3D fluid, for example in the case of a self-similar solution with parabolic similarity variable $\xi=x/t^\frac12$. Moreover, with $\alpha=5/2$, the critical Sobolev space for the dyadic model becomes $H^\frac12$, in agreement with the PDE setting.
\end{remark}

\begin{remark}
\label{rem:DNO-typeII}
The scaling-critical Sobolev exponent of \eqref{DNO} is $s_c=\alpha-2$, so
the blow-up of Theorem~\ref{th:DNO} is of Type~I or Type~II according to whether
$s_0>s_c$ or $s_0<s_c$. Since $s_0<\tfrac12$ and $s_c=\tfrac12$ at
$\alpha=\tfrac52$, the theorem is just \emph{slightly} Type~II. Equivalently, the critical norm $\sup_kN_k^{\alpha-2}|X_k|$ grows by a
factor $2^{\alpha-2}R$ per shell, which for the computed solution at
$\alpha=2.4906$, gives
\[
 \sup_kN_k^{\alpha-2}|X_k(t)|\sim(T-t)^{-0.020},
 \qquad t\to T-.
\]
This should be contrasted with the inverse cascade of
Theorem~\ref{th:alpha>2}, where the same quantity grows like
$t^{-(\alpha-2)/\alpha}$ (Figure~\ref{fig:type1-type2}).
\end{remark}

\begin{figure}[ht]
\centering
\makebox[\textwidth][c]{
\includegraphics[width=1.1\textwidth]{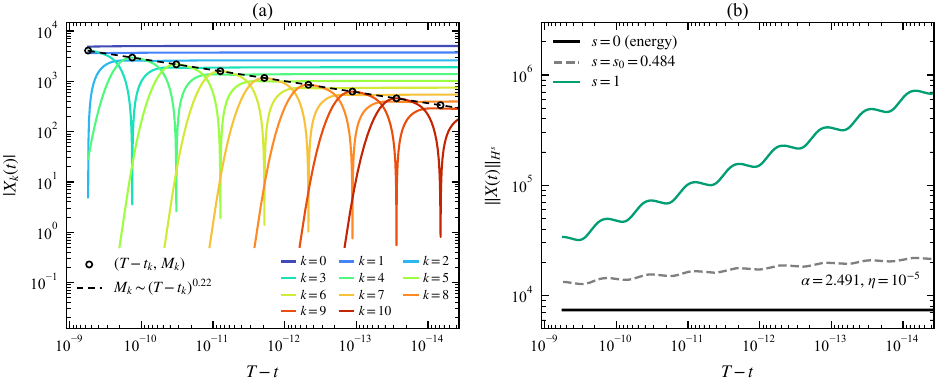}}
\caption{The blow-up cascade generated by the finite-support data of
Theorem~\ref{th:DNO} at $\alpha=2.4906$. (a)~Each curve is
one shell $|X_k|$.  The downward spikes are the sign changes at
the shell-to-shell transfers, after which a mode becomes negative and remains approximately frozen. The circles mark the entry of each shell, $(T-t_k,M_k)$. The dashed line shows the exponentially decaying
transfer times and energies as $t$ approaches the blow-up. (b)~The resulting Sobolev norms, with $s_0=0.4838$ the threshold of Theorem~\ref{th:DNO}. Above $s_0$ the norm grows geometrically in the scale. (The simulation in fact shows a slightly better threshold $s>0.4492$.)}
\label{fig:cascade-mixed}
\end{figure}

We explain in the following the main mechanism of the blow-up construction
through the mixed DN and Obukhov cascade; the full proof of Theorem~\ref{th:DNO} will be given in a forthcoming companion paper. There are two known regularization mechanisms for dyadic models, and it is remarkable that \eqref{DNO} is able to evade both without any forcing. First, for a pure DN model, there is the regularization effect from turbulent cascade for $\alpha\leq5/2$ proved in \cite{MR2844828}, and even for $\alpha\leq 3$ (when $\lambda$ is large enough) recently proved in \cite{arXiv:2609.21312}, reaching the blow-up threshold of \cite{MR2415066}. For a pure Obukhov model, there is regularization due to excessively high amplitudes in the cascade, discovered in \cite{MR2180809}. A large amplitude leading the cascade causes the inverse cascade term (that is, $-N_k^\alpha X_{k+1}^2$) to deplete the trailing mode until it becomes negative; these negative modes eventually become strong enough to suspend the cascade. Finite-time blow-ups demonstrated previously avoid these obstacles either by considering a highly engineered model \cite{MR3486169} or by introducing a force \cite{palasek2026finite}.

The blow-up is asymptotically discretely self-similar, so it suffices to explain the energy transfer between a particular pair of modes. As a first approximation, assume the interaction between $X_k$ and $X_{k+1}$ begins with $X_k\sim M$ and $X_{k+1}=0$. The main growth is exponential and results from the forward Obukhov term
$N_k^\alpha X_kX_{k+1}$ in the equation for $X_{k+1}$. Being linear in
$X_{k+1}$, this term can only amplify once $X_{k+1}$ is ``nudged'' away from zero.\footnote{One might imagine this role could be played instead by some small but non-zero initial data for $X_{k+1}$. This strategy fails due to the viscosity acting over the full lifetime of the solution. In the inviscid case, there is an obstacle in the form of the Kiselev--Zlato\v s regularity mechanism: $X_{k+1}(0)$ must be very small for the data to be smooth, but then the energy transfer takes too long, leading to modes becoming negative.} This nudge of $X_{k+1}$ is enacted by the forward DN cascade term $\eta N_{k}^\alpha X_{k}^2$.

Roughly speaking, energy moves to higher modes, with each transfer retaining a sufficient fraction of the active amplitude and the transfer times decreasing geometrically. To isolate the roles of the two types of nonlinearities, we restrict attention to the interactions of two neighboring shells. Fix index $k$ and let
\[
y:=X_k, \qquad z:=X_{k+1},\qquad \mu =N_k^\alpha
\]
Ignoring the dissipation and interaction with other modes, we have
\[
y'=-\mu z^2-\eta\mu yz, \qquad z'=\mu yz+\eta\mu y^2
\]
which conserves $y^2+z^2$ exactly. The transfer is therefore a rotation
along the circle $y^2+z^2=E$, and blow-up occurs if a definite fraction of
the energy is transferred before the neighboring shells significantly interfere. Starting from $y=M>0$ and
$z=0$, the Obukhov forward cascade term $\mu yz$ cannot cause any growth; instead it is the DN forward cascade
term $\eta\mu y^2$ that lifts $z$ off zero. Once $z>0$, the Obukhov terms
amplify $z$ exponentially while depleting $y$, and the DN terms begin
preparing the next shell. The ratio of the DN source to the Obukhov
amplification in the $z$ equation is
\[
\frac{\eta\mu y^2}{\mu yz}=\frac{\eta y}{z},
\]
so the DN term matters only while $z\lesssim\eta y$; thereafter Obukhov
amplification dominates. The resulting dynamics are shown in
Figure~\ref{fig:cascade-mixed}.

The DN and Obukhov terms work together to produce a frequency-localized cascade, resembling Tao's delay mechanism~\cite{MR3486169}. This localization prevents the energy from spreading in frequency, which is a known regularizing effect of the nonlinear term~\cite{MR3415578}. Figure~\ref{fig:DN-comparison} illustrates the difference: for the pure DN model the energy disperses across scales until the dissipation takes over, whereas the mixed cascade keeps a sharp front and carries higher Sobolev norms to large frequencies. The result is blow-up in the 3D Navier--Stokes regime $\alpha=5/2$, where the DN and Obukhov models are each known to be regular when considered separately.

\begin{figure}[ht]
\centering
\makebox[\textwidth][c]{
\includegraphics[width=1.1\textwidth]{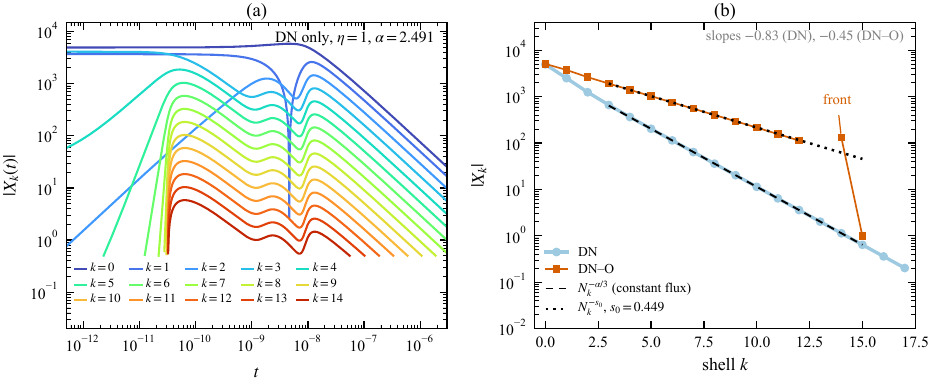}}
\caption{(a)~The turbulent cascade for the DN model. Without a delay mechanism, energy disperses across scales until the dissipation takes over. (b)~A comparison of the energy spectrum for the DN and the DN-O models. The spectrum is steeper for DN, with no sharp front as in the mixed DN-O case. The DN-O cascade is able to carry higher norms to large frequencies, meaning it can withstand larger dissipation.}
\label{fig:DN-comparison}
\end{figure}

Our interest in the DN-Obukhov model is that its blow-up mechanism partially resembles the one recently demonstrated by OpenAI for the 3D forced Navier--Stokes equations in \cite{openai2026}. Consecutive shells may be viewed as vortex components at successively finer spatial scales.  
The DN interaction in the dyadic model corresponds to self-similar creation of finer vortex components, while the Obukhov interaction captures amplification of the pulses in \cite{openai2026}. In this process, nonlinear transfer occurs from a coarser background structure into a finer background structure, canceling errors from the self-similar progression of vortex concentration through scales. In this sense, the construction in \cite{openai2026} combines DN-type cascade with an Obukhov-type instability whose stress feedback helps sustain the core vortex background. Therefore, our dyadic model captures aspects of self-similar cascade through the DN interaction, and pulse amplification and convex integration-like cancellation through the Obukhov nonlinearity. 
We acknowledge that the dyadic model suppresses the spatial structure of the Navier--Stokes equations. 
We note that the force plays a prominent role in \cite{openai2026}, not least by creating the pulses, while there is no external forcing in our dyadic setting.

\textbf{Organization of the paper.} The remainder of the
paper concerns the inverse cascade. Section~\ref{sec:first_approach} proves
Theorems~\ref{th:alpha<2}, \ref{th:critical-blowup} and \ref{th:alpha>2} by a
Lyapunov function argument, Section~\ref{sec:second_approach} outlines the
alternative multiscale construction in the supercritical case, and
Section~\ref{s:inviscid} treats the inviscid model, proving
Theorem~\ref{th:inviscid} and Corollary~\ref{cor:inviscid-forward-blowup}.

\section{First approach: Galerkin approximation with energy flux and a Lyapunov function}\label{sec:first_approach}

We begin with the following elementary observation.

\begin{lemma}\label{l:positivity}
If $X_k(0)\geq 0$ for all $k\geq 1$, the components of the solution to \eqref{Obukhov} remain non-negative for all time.  
\end{lemma}
\begin{proof}
For each $k$, rewrite \eqref{Obukhov} as
\[
X_k'+\bigl(\nu N_k^2+N_{k-1}^\alpha X_{k-1}\bigr)X_k
=N_k^\alpha X_{k+1}^2.
\]
The variation-of-constants formula preserves non-negativity, since both the
initial value and the forcing on the right-hand side are non-negative.
\end{proof}

We consider the following finite-dimensional Galerkin approximation with energy flux. Let $F_K\geq 0$. For $K\geq 1$, let $X^K=(X_1^K,\dots,X_{K+1}^K)$ be the solution to
\begin{equation}\label{eq:galerkin}
\begin{cases}
\displaystyle
\frac{d}{dt}X_k^K
=
-\nu N_k^2X_k^K
+N_k^\alpha (X_{k+1}^K)^2
-N_{k-1}^\alpha X_{k-1}^KX_k^K,
&1\leq k\leq K+1,\\[0.4em]
X_0^K=X_{K+2}^K=0,\\
X_k^K(0)=F_K\delta_{k,K+1}
\end{cases}
\end{equation}
where $\delta_{i,j}$ denotes the Kronecker delta. By Lemma~\ref{l:positivity}, $X_k^K(t)\geq0$ for all $k$. The nonlinear terms telescope, giving
\begin{equation}\label{Galerkin-energy-general}
\|X^K(t)\|_{\ell^2}^2
+
2\nu\int_0^t
\sum_{k=1}^{K+1}N_k^2(X_k^K(\tau))^2\,d\tau
=
F_K^2.
\end{equation}
In particular, the Galerkin solution exists globally.

The main ingredient in the proof below is a rapidly decreasing Lyapunov function. For $0<q<\alpha$, define
\begin{equation}\label{def-HqK}
H_{q,K}(t)
=
\sum_{k=1}^{K+1}N_k^{2q}(X_k^K(t))^2
-c\sum_{k=1}^{K}N_k^{2q}X_k^K(t)X_{k+1}^K(t),
\end{equation}
where $c>0$ will be chosen sufficiently small in the following lemma.
We set
\[
p=\alpha + 2q.
\]

Direct computations yield the following.

\begin{lemma}\label{le-Lyapunov}
Let $0<q<\alpha$, $T>0$ and $X_k^K(t)\geq 0$ for all $1\leq k\leq K+1$. Suppose that
\[
\sup_K\sup_{0\leq t\leq T}X_1^K(t)\leq M.
\]
For $c>0$ sufficiently small, depending on $\alpha,\lambda,\nu,q,M$ but not on
$K$, one has
\[
H_{q,K}(t)\geq0,
\qquad
H_{q,K}'(t)\leq - C H_{q,K}(t)^{3/2},
\qquad 0\leq t\leq T,
\]
where $C=C(\alpha,\lambda,\nu,q,M)>0$ is independent of $K$.
\end{lemma}
\begin{proof}
We omit the superscript $K$. Recall that $N_k=\lambda^k$ for $k\geq0$.
By the Cauchy--Schwarz inequality,
\begin{equation*}
\begin{split}
c\sum_{k=1}^{K} N_k^{2q}X_k(t)X_{k+1}(t)
&=\ c\lambda^{-q}
\sum_{k=1}^{K}\left(N_k^{q}X_k(t)\right)
\left(N_{k+1}^{q}X_{k+1}(t)\right)\\
&\leq\ \frac12c\lambda^{-q}
\left(\sum_{k=1}^{K}N_k^{2q}X_k^2(t)+\sum_{k=1}^{K}N_{k+1}^{2q}X_{k+1}^2(t)\right)\\
&\leq\ c\lambda^{-q} \sum_{k=1}^{K+1}N_k^{2q}X_k^2(t).
\end{split}
\end{equation*}
Thus it follows from \eqref{def-HqK} that
\[
\sum_{k=1}^{K+1}N_k^{2q}X_k^2(t)
\geq H_{q,K}(t)
\geq
(1-c\lambda^{-q})
\sum_{k=1}^{K+1}N_k^{2q}X_k^2(t),
\]
provided $c>0$ is small enough. In particular,
\begin{equation}\label{H-norm-K}
H_{q,K}(t)\approx \sum_{k=1}^{K+1}N_k^{2q}X_k^2(t),
\end{equation}
with constants independent of $K$.

By H\"older's inequality,
\begin{equation}\label{H-norm-2K}
\begin{split}
\sum_{k=1}^{K+1}N_k^{2q}X_k^2(t)
=&\sum_{k=1}^{K+1} N_k^{\frac23(q-\alpha)}
\left(N_k^{\frac23(\alpha+2q)}X_k^2(t) \right)\\
\leq& \left( \sum_{k=1}^{K+1} N_k^{2(q-\alpha)}\right)^{\frac13}
\left(\sum_{k=1}^{K+1} N_k^{\alpha+2q}X_k^3(t)\right)^{\frac23}\\
\lesssim& \left(\sum_{k=1}^{K+1} N_k^pX_k^3(t)\right)^{\frac23},
\end{split}
\end{equation}
since $0<q<\alpha$.

Employing the Galerkin equation gives
\begin{equation}\label{est-H-1K}
\begin{split}
\frac{d}{dt}\sum_{k=1}^{K+1} N_k^{2q}X_k^2(t)&
=-2\nu \sum_{k=1}^{K+1} N_k^{2+2q}X_k^2(t)
+2\sum_{k=1}^{K} N_k^{\alpha+2q}X_k(t) X_{k+1}^2(t)\\
&\qquad-2\sum_{k=1}^{K+1}N_{k-1}^\alpha N_k^{2q}X_{k-1}(t)X_k^2(t)\\
&=-2\nu \sum_{k=1}^{K+1} N_k^{2+2q}X_k^2(t)
+2(1-\lambda^{2q})\sum_{k=1}^{K} N_k^pX_k(t) X_{k+1}^2(t).
\end{split}
\end{equation}
Moreover,
\begin{equation}\label{est-H-2K}
\begin{split}
\frac{d}{dt}\sum_{k=1}^{K} &N_k^{2q}X_k(t)X_{k+1}(t)\\
=&-\nu (1+\lambda^2)\sum_{k=1}^{K} N_k^{2+2q}X_k(t)X_{k+1}(t)\\
&+\sum_{k=1}^{K} N_k^pX_{k+1}^3(t)
-\sum_{k=1}^{K}N_{k-1}^\alpha N_k^{2q}X_{k-1}(t)X_k(t)X_{k+1}(t)\\
&+\lambda^\alpha\sum_{k=1}^{K} N_k^pX_k(t)X_{k+2}^2(t)
-\sum_{k=1}^{K} N_k^pX_k^2(t)X_{k+1}(t).
\end{split}
\end{equation}
Here the terms containing $X_0$ and $X_{K+2}$ vanish. Combining \eqref{est-H-1K} and \eqref{est-H-2K}, we obtain
\begin{equation}\label{est-H-3K}
\begin{split}
- H_{q,K}'(t)=&\ 2\nu \sum_{k=1}^{K+1} N_k^{2+2q}X_k^2(t)
-c\nu (1+\lambda^2)\sum_{k=1}^{K} N_k^{2+2q}X_k(t)X_{k+1}(t)\\
&+2(\lambda^{2q}-1)\sum_{k=1}^{K} N_k^pX_k(t) X_{k+1}^2(t)\\
&+c\sum_{k=1}^{K} N_k^pX_{k+1}^3(t)
+c\lambda^\alpha\sum_{k=1}^{K} N_k^pX_k(t)X_{k+2}^2(t)\\
&-c\sum_{k=1}^{K}N_{k-1}^\alpha N_k^{2q}X_{k-1}(t)X_k(t)X_{k+1}(t)
-c\sum_{k=1}^{K} N_k^pX_k^2(t)X_{k+1}(t).
\end{split}
\end{equation}
The positive cubic terms control the three adverse cubic terms, while the
mixed viscous term is absorbed by the viscous dissipation. Indeed,
\begin{equation*}
\begin{split}
c\nu (1+\lambda^2)&\sum_{k=1}^{K} N_k^{2+2q}X_k(t)X_{k+1}(t)\\
=&\ c\nu (1+\lambda^2)\lambda^{-1-q}
\sum_{k=1}^{K}\left(N_k^{1+q}X_k(t)\right)
\left(N_{k+1}^{1+q}X_{k+1}(t)\right)\\
\leq&\ c\nu (1+\lambda^2)\lambda^{-1-q}
\sum_{k=1}^{K+1}N_k^{2+2q}X_k^2(t).
\end{split}
\end{equation*}
Similarly, we estimate
\begin{equation*}
\begin{split}
c\sum_{k=1}^{K}&N_{k-1}^\alpha N_k^{2q}X_{k-1}(t)X_k(t)X_{k+1}(t)\\
=&\ c\sum_{k=1}^{K}\left(\lambda^{2q-\frac12\alpha}N_{k-1}^{\frac12\alpha+q}X_{k-1}^{\frac12}(t)X_k(t)\right)
\left(\lambda^{\frac12\alpha}N_{k-1}^{\frac12\alpha+q}X_{k-1}^{\frac12}(t)X_{k+1}(t)\right)\\
\leq &\ \frac12c\lambda^{4q-\alpha}
\sum_{k=1}^{K} N_{k-1}^p X_{k-1}(t)X_k^2(t)
+\frac12c \lambda^{\alpha}\sum_{k=1}^{K} N_{k-1}^p X_{k-1}(t)X_{k+1}^2(t)\\
\leq &\ \frac12c\lambda^{4q-\alpha}
\sum_{k=1}^{K} N_k^p X_k(t)X_{k+1}^2(t)
+\frac12c \lambda^{\alpha}\sum_{k=1}^{K} N_k^p X_k(t)X_{k+2}^2(t),
\end{split}
\end{equation*}
and
\begin{equation*}
\begin{split}
c\sum_{k=1}^{K}& N_k^pX_k^2(t)X_{k+1}(t)\\
=&\ c\sum_{k=1}^{K}\left(\lambda^{\frac12\alpha+q}N_k^{\frac12\alpha+q}X_k^{\frac12}(t)X_{k+1}(t)\right)
\left(N_{k-1}^{\frac12\alpha+q}X_k^{\frac32}(t)\right)\\
\leq &\ \frac12c\lambda^{\alpha+2q}\sum_{k=1}^{K} N_k^pX_k(t)X_{k+1}^2(t)
+\frac12c\sum_{k=1}^{K}N_{k-1}^pX_k^3(t)\\
\leq &\ \frac12c\lambda^{\alpha+2q}\sum_{k=1}^{K} N_k^pX_k(t)X_{k+1}^2(t)
+\frac12c\sum_{k=1}^{K}N_k^pX_{k+1}^{3}(t)+CcX_1^3(t).
\end{split}
\end{equation*}
Putting the estimates above together with \eqref{est-H-3K} yields
\begin{equation}\label{est-H-4K}
\begin{split}
- H_{q,K}'(t)\geq &\ \left(2\nu-c\nu(1+\lambda^2)\lambda^{-1-q}\right)
\sum_{k=1}^{K+1} N_k^{2+2q}X_k^2(t)\\
&+\left(2(\lambda^{2q}-1)-\frac12c\lambda^{4q-\alpha}-\frac12c\lambda^{\alpha+2q}\right)
\sum_{k=1}^{K} N_k^pX_k(t) X_{k+1}^2(t)\\
&+\frac12c\sum_{k=1}^{K} N_k^pX_{k+1}^3(t)
+\frac12c\lambda^\alpha\sum_{k=1}^{K} N_k^pX_k(t)X_{k+2}^2(t)
-CcX_1^3(t)\\
\geq &\ \nu \sum_{k=1}^{K+1} N_k^{2+2q}X_k^2(t)
+\frac12c\sum_{k=1}^{K} N_k^pX_{k+1}^3(t)
-CcX_1^3(t),
\end{split}
\end{equation}
provided $c>0$ is small enough so that
\begin{equation*}
\begin{split}
2\nu-c\nu(1+\lambda^2)\lambda^{-1-q}&>\nu,\\
2(\lambda^{2q}-1)-\frac12c\lambda^{4q-\alpha}-\frac12c\lambda^{\alpha+2q}&>0.
\end{split}
\end{equation*}
In the last line we used positivity. By assumption, $X_1(t)\leq M$ on
$[0,T]$, and hence
\[
X_1^3(t)\leq M X_1^2(t)
\lesssim_M \sum_{k=1}^{K+1} N_k^{2+2q}X_k^2(t).
\]
Decreasing $c$ if necessary, depending also on $M$, we obtain
\begin{equation}\label{est-H-5K}
-H_{q,K}'(t)\gtrsim \sum_{k=1}^{K+1} N_k^{2+2q}X_k^2(t)
+\sum_{k=1}^{K} N_k^pX_{k+1}^3(t).
\end{equation}
Moreover,
\[
\sum_{k=2}^{K+1}N_k^pX_k^3(t)
=
\lambda^p\sum_{k=1}^{K}N_k^pX_{k+1}^3(t),
\qquad
N_1^pX_1^3(t)\lesssim_M
\sum_{k=1}^{K+1}N_k^{2+2q}X_k^2(t).
\]
It then follows from \eqref{est-H-5K}, \eqref{H-norm-2K}, and \eqref{H-norm-K} that
\[
-H_{q,K}'(t)\gtrsim \left(\sum_{k=1}^{K+1} N_k^{2q}X_k^2(t)\right)^{3/2}
\gtrsim H_{q,K}(t)^{3/2}.
\]
\end{proof}

Whenever $H_{q,K}(0)>0$, integration of the differential inequality in
Lemma~\ref{le-Lyapunov} gives
\begin{equation}\label{Lyapunov-decay}
H_{q,K}(t)
\lesssim
\bigl(t+H_{q,K}(0)^{-1/2}\bigr)^{-2}.
\end{equation}
If the hypothesis on $X_1^K$ holds globally in time, then
\eqref{Lyapunov-decay} holds for every $t\geq0$.

\subsection{The energy-supercritical case}
\label{ss:supercritical}
In the case $\alpha>2$, the nonlinear effect dominates the linear one and hence it is possible to design a mechanism of complete energy transfer from high modes to low modes. The proof of the finite energy jump stated in Theorem~\ref{th:alpha>2} is more approachable compared to the other two cases.
\begin{proof}[Proof of Theorem~\ref{th:alpha>2}]

Take
\[
q:=\frac{\alpha+2}{2}\in(2,\alpha), \qquad
\gamma:=1-\frac2q=\frac{\alpha-2}{\alpha+2}>0.
\]
For each $K$, let $X^K$ be the solution to the Galerkin approximation \eqref{eq:galerkin} with $F_K=1$.
By Lemma~\ref{l:positivity} and energy balance \eqref{Galerkin-energy-general},
\[
0\leq X_k^K(t)\leq \|X^K(t)\|_{\ell^2}\leq1, \qquad t\geq0.
\]
Fix the coefficient $c>0$ in $H_{q,K}$ as given by Lemma~\ref{le-Lyapunov}. The lemma applies globally, with constants independent of $K$.
Since $H_{q,K}(0)=N_{K+1}^{2q}$, \eqref{H-norm-K} and
\eqref{Lyapunov-decay} imply
\[
\|X^K(t)\|_{H^q}^2\lesssim t^{-2},
\qquad t>0.
\]
Hence, by the weighted H\"older inequality, 
\[
\|X^K(t)\|_{H^1}^2 \leq \|X^K(t)\|_{H^q}^{2/q} \|X^K(t)\|_{\ell^2}^{2-2/q}
\lesssim t^{-2/q}.
\]
Since $q>2$, this estimate is integrable near the origin, so from the energy equality we get
\[
1-\|X^K(t)\|_{\ell^2}^2 = 2\nu\int_0^t\|X^K(\tau)\|_{H^1}^2\,d\tau
\lesssim \int_0^t\tau^{-2/q}\,d\tau \lesssim t^\gamma.
\]
Thus
\begin{equation}\label{supercritical-energy-ret}
1-Ct^\gamma \leq \|X^K(t)\|_{\ell^2}^2 \leq 1,
\qquad t\geq 0,
\end{equation}
where the constant $C$ is independent of $K$. 

For any $L\geq2$,
\[
\frac12\frac{d}{dt} \sum_{k=L}^{K+1}(X_k^K(t))^2 = -\nu\sum_{k=L}^{K+1}N_k^2(X_k^K)^2
-N_{L-1}^\alpha X_{L-1}^K (X_L^K)^2.
\]
Since the flux term is non-positive,
\begin{equation}\label{est-high-K}
\sum_{k=L}^{K+1}(X_k^K(t))^2\leq e^{-2\nu N_L^2t}\sum_{k=L}^{K+1}(X_k^K(0))^2\leq e^{-2\nu N_L^2t}.
\end{equation}

We now construct the limit solution $X(t)$. For each fixed number of modes, energy balance \eqref{Galerkin-energy-general} gives uniform
bounds on the components and their derivatives on every compact time
interval. Then using the Arzel\`a--Ascoli  theorem and passing to a diagonal subsequence, we infer that $X^K \to X$ in $C_w([0,1];\ell^2)$ for some weakly continuous in $\ell^2$ global solution of \eqref{Obukhov} $X(t)$. In other words, 
\begin{equation} \label{eq:Galerkin-conv}
X_k^K\to X_k \quad\text{locally uniformly on }[0,\infty)
\end{equation}
for every $k$ for $X_k \in C([0,\infty)$ satisfying \eqref{Obukhov}.

For every fixed $k$ and
$K\geq k$, positivity and energy balance \eqref{Galerkin-energy-general} imply
\[
(X_k^K)'\leq N_k^\alpha (X_{k+1}^K)^2\leq N_k^\alpha,
\qquad X_k^K(0)=0.
\]
Thus $X_k^K(t)\leq N_k^\alpha t$, and, passing to the limit, we obtain
\[
0 \leq X_k(t)\leq N_k^\alpha t.
\]
In particular, $X_k(t)\to0$ as $t\to0+$. Moreover, by \eqref{Lyapunov-decay}, for every $M$ and $t>0$,
\[
\sum_{k=1}^M X_k(t)^2 =
\lim_{K\to\infty}\sum_{k=1}^M (X_k^K(t))^2 \leq1,
\]
so $\|X(t)\|_{\ell^2}^2\leq1$.

Now for fixed $L\geq2$ and $t>0$, we split the energy of the Galerkin solution as
\[
\sum_{k=1}^{L-1}(X_k^K(t))^2
=\|X^K(t)\|_{\ell^2}^2-\sum_{k=L}^{K+1}(X_k^K(t))^2.
\]
By \eqref{supercritical-energy-ret} and \eqref{est-high-K}, we have
\[
\sum_{k=1}^{L-1}(X_k^K(t))^2
\geq 1-Ct^\gamma-e^{-2\nu N_L^2t}.
\]
Now let $X(t)$ be the limit obtained above (see \eqref{eq:Galerkin-conv}). Then
\[
\sum_{k=1}^{L-1}X_k(t)^2 \geq 1-Ct^\gamma-e^{-2\nu N_{L}^2t},
\qquad t>0.
\]
Letting $L\to\infty$ followed by $t\to0+$, we obtain
\[
\liminf_{t\to0+}\|X(t)\|_{\ell^2}^2\geq1.
\]
Together with $\|X(t)\|_{\ell^2}^2\leq1$, this yields
\[
\lim_{t\to0+}\|X(t)\|_{\ell^2}=1.
\]
For $s<0$ and $L\geq2$,
\[
\|X(t)\|_{H^s}^2
\leq \sum_{k<L}N_k^{2s}X_k(t)^2+N_L^{2s}\|X(t)\|_{\ell^2}^2.
\]
First let $t\to0+$ and then $L\to\infty$ to obtain
\[
\lim_{t\to0+}\|X(t)\|_{H^s}=0.
\]
Finally, for each $M\geq L$, \eqref{eq:Galerkin-conv} and \eqref{est-high-K} imply
\[
\sum_{k=L}^{M}X_k(t)^2 = \lim_{K\to\infty}\sum_{k=L}^{M}(X_k^K(t))^2
\leq e^{-2\nu N_L^2t}.
\]
Taking $M\to\infty$, we obtain
\[
\sum_{k\geq L}X_k(t)^2\leq e^{-2\nu N_L^2t},
\qquad L\geq2,
\]
for every $t>0$. Therefore, for every $s>0$ we obtain the smallness of $H^s$-tails:
\[
\sum_{k\geq L}N_k^{2s}X_k(t)^2 \lesssim_s
\sum_{j\geq L}N_j^{2s}\sum_{k\geq j}X_k(t)^2
\lesssim_s \sum_{j\geq L}N_j^{2s}e^{-2\nu N_j^2t}.
\]
In particular $X(t)\in H^s$ for every $s\geq0$ and $t>0$. Moreover, by continuity of $X_k(t)$ for each $k$, we have
\[
X\in C((0,\infty);H^s), \qquad s \in \mathbb{R}.
\]
\end{proof}

\begin{remark}[Sobolev size as $t\to0+$]\label{rem-Hs-rates}
For every $2<\sigma<\alpha$, the Lyapunov estimate gives
\[
\|X(t)\|_{H^\sigma}\lesssim_\sigma t^{-1}.
\]
Interpolating with $\|X(t)\|_{\ell^2}\leq1$ and using 
\[
\sum_{k\geq L}X_k^2(t)\leq e^{-2\nu N_L^2t},
\]
one obtains, for every $\varepsilon>0$,
\[
\|X(t)\|_{H^s} \lesssim_{s,\varepsilon}
\begin{cases}
t^{-s/\alpha-\varepsilon},
&0<s\leq\alpha,\\[0.3em]
t^{-(s-\alpha+2)/2-\varepsilon},
&s>\alpha.
\end{cases}
\]

Conversely, the estimate $X_k(t)\leq N_k^\alpha t$, together with
$\|X(t)\|_{\ell^2}\to1$, shows that a fixed fraction of the energy lies at
frequencies $N_k\gtrsim t^{-1/\alpha}$. Hence, for every $s>0$,
\begin{equation}\label{Hs-lower}
\|X(t)\|_{H^s}\gtrsim_s t^{-s/\alpha},
\qquad t\to0+.
\end{equation}
Thus the bounds are sharp up to the arbitrarily small loss
$t^{-\varepsilon}$ when $0<s\leq\alpha$. For $s>\alpha$, the present
argument leaves a gap.\footnote{Non-rigorous numerics indicate the stronger bound $t^{-s/\alpha}$ holds even for $s>\alpha$, suggesting this gap is merely an artifact of the proof.}
\end{remark}

\subsection{The energy-subcritical case}
\label{ss:subcritical}

Let $1\leq\alpha<2$ and denote
\[
\beta=2-\alpha>0.
\]
We assume in this subsection that $\lambda^\beta>2$. Take
\[
a=\lambda^{2\beta}, \qquad b=\lambda^{-2}.
\]
For $A>0$, let $X^{K,A}$ denote the solution of
\eqref{eq:galerkin} with
\[
F_K=A N_{K+1}^{\beta},
\]
and define
\[
Y_k^{K,A}=N_k^{-\beta}X_k^{K,A}.
\]
Then, with $Y_0^{K,A}=Y_{K+2}^{K,A}=0$, we have
\begin{equation}\label{normalized-subcritical}
(Y_k^{K,A})' = N_k^2\left(-\nu Y_k^{K,A} + a(Y_{k+1}^{K,A})^2
-bY_{k-1}^{K,A}Y_k^{K,A}\right)
\end{equation}
and $Y_{K+1}^{K,A}(0)=A$.

The following elementary estimate shows an inverse energy cascade mechanism: a sufficiently large normalized amplitude cannot remain trapped in shell \(j+1\); it must quickly propagate toward lower shells.
This lemma is the only step that uses the shell separation.

\begin{lemma}\label{le-one-shell-transfer} 
There are constants $B,C_*>0$, depending only on
$\alpha,\lambda,\nu$, with the following property. Let $Y$ be a non-negative
solution of \eqref{normalized-subcritical}. Suppose that, for some
$1\leq j\leq K$ and $t_0\geq0$,
\[
Y_{j+1}(t_0)=B,
\]
and $Y_{j-1}<B$ on $[t_0,t_0+C_*N_j^{-2}]$. Then $Y_j(t)\geq B$ for some
$t\in[t_0,t_0+C_*N_j^{-2}]$.
\end{lemma}

\begin{proof}
Take
\[
\eta_\lambda=(2\lambda^2)^{-\frac{2\lambda^2}{2\lambda^2-1}}.
\]
Since $(2\lambda^2)^{1/(2\lambda^2-1)}\leq2$, we have
$\eta_\lambda\geq(4\lambda^2)^{-1}$. The assumption
$\lambda^\beta>2$ therefore implies $a\eta_\lambda>b$. Choose $B>0$ such that
\begin{equation}\label{choice-B}
(a\eta_\lambda-b)B>\nu.
\end{equation}
Let $d=\nu+bB$. If the conclusion fails, then $Y_j<B$ on the whole
interval $[t_0,t_0+C_*N_j^{-2}]$. Then, by positivity and \eqref{normalized-subcritical}, we have
\[
Y_{j+1}(t_0+s)\geq B e^{-\lambda^2dN_j^2s},
\]
and hence
\[
Y_j'(t_0+s) \geq N_j^2\left(-dY_j(t_0+s)
+aB^2e^{-2\lambda^2dN_j^2s}\right).
\]
Therefore
\[
Y_j(t_0+s) \geq \frac{aB^2}{d(2\lambda^2-1)}
\left(e^{-dN_j^2s}-e^{-2\lambda^2dN_j^2s}\right).
\]
The expression in parentheses attains the value
$(2\lambda^2-1)\eta_\lambda$ at $s=C_*N_j^{-2}$, where
\[
C_*=\frac{\log(2\lambda^2)}{(2\lambda^2-1)d}.
\]
By \eqref{choice-B}, the resulting lower bound is larger than $B$, a
contradiction.
\end{proof}

In view of the inverse cascade mechanism revealed in the previous lemma, we further show that: starting with data concentrated in the remote shell \(K+1\), one can tune its amplitude so that a uniformly sized normalized pulse travels through every lower shell, without any normalized component becoming unbounded. Namely, we have the following ``shooting'' result.
\begin{proposition}\label{prop-critical-flux}
Fix $q\in(0,\alpha)$. There are constants
$A_-,B,m_0,c_0,C_0>0$ and amplitudes $A_K\in[A_-,B]$ such that the
solutions of \eqref{eq:galerkin} with
\begin{equation}\label{critical-flux-data}
F_K=A_KN_{K+1}^{2-\alpha}
\end{equation}
satisfy
\begin{equation}\label{critical-uniform-bound}
0\leq N_k^{\alpha-2}X_k^K(t)\leq B,
\qquad \forall 1\leq k\leq K+1,\ t\geq0,
\end{equation}
and
\begin{equation}\label{critical-first-shell}
\sup_{t\geq0}N_1^{\alpha-2}X_1^K(t)=\frac B2.
\end{equation}
Moreover, for every $1\leq k\leq K$,
\begin{equation}\label{critical-flux-bound}
\sup_{c_0N_k^{-2}\leq t\leq C_0N_k^{-(q+2-\alpha)}}
N_k^{\alpha-2}X_k^K(t)
\geq m_0.
\end{equation}
All constants are independent of $K$.
\end{proposition}

\begin{proof}
Recall that $\beta=2-\alpha$, and let $B$ be the threshold from
Lemma~\ref{le-one-shell-transfer}. For $A>0$, let $X^{K,A}$ be the Galerkin
solution with $F_K=AN_{K+1}^{\beta}$, and write
\[
Y_k^{K,A}=N_k^{-\beta}X_k^{K,A}.
\]
By Lemma~\ref{l:positivity}, all components of $Y^{K,A}$ are non-negative.

\smallskip
\noindent\emph{Shooting and the uniform barrier.}
By the energy equality we have
\[
\frac{d}{dt}\|X^{K,A}(t)\|_{\ell^2}^2
\leq -2\nu N_1^2\|X^{K,A}(t)\|_{\ell^2}^2,
\]
and hence
\begin{equation}\label{critical-energy-decay}
\|X^{K,A}(t)\|_{\ell^2}
\leq AN_{K+1}^{\beta}e^{-\nu N_1^2t}.
\end{equation}
Thus every $Y_k^{K,A}$ tends to zero as $t\to\infty$. For fixed $K$,
we use finite-time continuous dependence on $A$, together with \eqref{critical-energy-decay}, which shows that
$A\mapsto Y_1^{K,A}$ is continuous into $C_0([0,\infty))$ on compact intervals. Since
$f\mapsto\sup_{t\geq0}f(t)$ is $1$-Lipschitz in the supremum norm,
\[
\Phi_K(A):=\max_{t\geq0}Y_1^{K,A}(t)
\]
is continuous. More generally, let
\[
M_k(A):=\max_{t\geq0}Y_k^{K,A}(t),
\qquad 1\leq k\leq K+1.
\]
All these maxima are attained.

For $1\leq k\leq K$, by variation of constants in
\eqref{normalized-subcritical},
\[
Y_k^{K,A}(t) = aN_k^2\int_0^t \exp\!\left(-N_k^2\int_s^t
(\nu+bY_{k-1}^{K,A}(r))\,dr\right)
(Y_{k+1}^{K,A}(s))^2\,ds.
\]
By positivity,
\begin{equation}\label{critical-Duhamel-max}
M_k(A)\leq \frac a\nu M_{k+1}(A)^2, \qquad 1\leq k\leq K.
\end{equation}
Moreover,
\[
(Y_{K+1}^{K,A})'
=-N_{K+1}^2(\nu+bY_K^{K,A})Y_{K+1}^{K,A}\leq0,
\]
so $M_{K+1}(A)=A$. Take $D=a/\nu$. Iterating
\eqref{critical-Duhamel-max} yields
\begin{equation}\label{critical-small-shooting}
\Phi_K(A)=M_1(A) \leq D^{2^K-1}A^{2^K} = D^{-1}(DA)^{2^K}.
\end{equation}
Now we choose, independently of $K$,
\[
0<A_-<\min\left\{B,D^{-1},\left(\frac{B}{2D}\right)^{1/2}\right\}.
\]
Since $K\geq1$ and $DA_-<1$, \eqref{critical-small-shooting} implies
\begin{equation}\label{critical-shooting-low}
\Phi_K(A_-) \leq D^{-1}(DA_-)^{2^K} \leq DA_-^2 <\frac B2.
\end{equation}

We claim that, for every $A\in(0,B]$,
\begin{equation}\label{critical-barrier}
M_1(A)<B
\quad\to \quad
M_j(A)<B, \qquad \forall 1\leq j\leq K+1.
\end{equation}
Suppose otherwise, and let
\[
m:=\min\{j\in\{1,\dots,K+1\}:M_j(A)\geq B\}.
\]
Then $m\geq2$. Since $Y_j^{K,A}(0)=0$ for $j\leq K$ and
$Y_{K+1}^{K,A}(0)=A\leq B$, there is a first time $t_0\geq0$ such that
$Y_m^{K,A}(t_0)=B$. By the minimality of $m$,
\[
Y_{m-2}^{K,A}(t)<B, \qquad t\geq0,
\]
where $Y_0^{K,A}\equiv0$ if $m=2$. Lemma~\ref{le-one-shell-transfer},
applied with $j=m-1$, then forces $M_{m-1}(A)\geq B$, contradicting the
minimality of $m$. This implies \eqref{critical-barrier}.

For $A=B$, the top shell satisfies $M_{K+1}(B)=B$. Therefore
\eqref{critical-barrier} implies $\Phi_K(B)\geq B$. By
\eqref{critical-shooting-low} and the continuity of $\Phi_K$, there is
$A_K\in(A_-,B)$ such that
\begin{equation}\label{critical-shooting-normalization}
\Phi_K(A_K)=\frac B2.
\end{equation}
Denote
\[
X^K:=X^{K,A_K}, \qquad Y^K:=Y^{K,A_K}.
\]
Since $M_1(A_K)=B/2<B$, \eqref{critical-barrier} gives
$M_j(A_K)<B$ for every $j$. Consequently,
\[
0\leq Y_j^K(t)\leq B,
\qquad \forall 1\leq j\leq K+1,\ t\geq0,
\]
which implies \eqref{critical-uniform-bound}; and
\eqref{critical-shooting-normalization} is precisely
\eqref{critical-first-shell}.

\smallskip
\noindent\emph{Uniform transfer through all shells.}
Write $M_k=M_k(A_K)$ and define
\[
m_0:=\min\left\{\frac B2,D^{-1}\right\}.
\]
From \eqref{critical-Duhamel-max},
\[
M_{k+1}\geq\left(\frac{M_k}{D}\right)^{1/2}.
\]
Since $m_0\leq D^{-1}$, we have
$(m_0/D)^{1/2}\geq m_0$. Starting from $M_1=B/2\geq m_0$, by induction, we obtain 
\begin{equation}\label{critical-shell-maxima}
M_k\geq m_0,
\qquad 1\leq k\leq K+1.
\end{equation}
For $1\leq k\leq K$, let $t_{k,K}$ be the first time at which
\[
Y_k^K(t_{k,K})=m_0.
\]
This time exists by \eqref{critical-shell-maxima}, because $Y_k^K(0)=0$ and
$M_k$ is attained. Using \eqref{normalized-subcritical} and
$0\leq Y_j^K\leq B$, we obtain
\[
(Y_k^K)'\leq aB^2N_k^2.
\]
Therefore
\begin{equation}\label{critical-lower-transfer-time}
t_{k,K}\geq \frac{m_0}{aB^2}N_k^{-2}.
\end{equation}

\smallskip

\noindent\emph{Upper bound for the transfer time.}
The uniform barrier gives
\[
X_1^K(t)=N_1^\beta Y_1^K(t)\leq BN_1^\beta,
\qquad t\geq0.
\]
Fix the coefficient $c>0$ in the definition of $H_{q,K}$ according to
Lemma~\ref{le-Lyapunov}, with $M=BN_1^\beta$. Then
Lemma~\ref{le-Lyapunov} and \eqref{Lyapunov-decay} apply globally, with
constants independent of $K$. Since only the $(K+1)$st component is nonzero
at time zero,
\[
H_{q,K}(0)=A_K^2N_{K+1}^{2(q+\beta)}.
\]
Thus, with
\[
\tau_K:=(A_KN_{K+1}^{q+\beta})^{-1},
\]
we have
\begin{equation}\label{critical-H-decay}
H_{q,K}(t)\lesssim(t+\tau_K)^{-2}.
\end{equation}
At $t=t_{k,K}$, coercivity \eqref{H-norm-K} implies
\[
H_{q,K}(t_{k,K}) \gtrsim N_k^{2q}(X_k^K(t_{k,K}))^2
=m_0^2N_k^{2(q+\beta)}.
\]

Combining this with \eqref{critical-H-decay}, we obtain a constant $C_0>0$,
independent of $k$ and $K$, such that
\begin{equation}\label{critical-upper-transfer-time}
t_{k,K}+\tau_K\leq C_0N_k^{-(q+\beta)}.
\end{equation}
Let
\[
c_0:=\frac{m_0}{aB^2}.
\]
Since $\tau_K\geq0$, \eqref{critical-lower-transfer-time} and
\eqref{critical-upper-transfer-time} imply
\[
c_0N_k^{-2} \leq t_{k,K} \leq C_0N_k^{-(q+\beta)}.
\]
At this time,
\[
N_k^{\alpha-2}X_k^K(t_{k,K}) =N_k^{-\beta}X_k^K(t_{k,K})
=Y_k^K(t_{k,K}) = m_0.
\]
Since $q+\beta=q+2-\alpha$, this is exactly
\eqref{critical-flux-bound}. All constants are independent of $K$.

\end{proof}

Now we proceed to

\begin{proof}[Proof of Theorem~\ref{th:alpha<2}]
Take $q=\alpha/2$ and let $X^K$ be the sequence supplied by
Proposition~\ref{prop-critical-flux}. By
\eqref{Lyapunov-decay},
\[
H_{q,K}(t)\lesssim(t+\tau_K)^{-2}.
\]
Since $A_K\geq A_->0$ and $q+\beta>0$,
\[
\tau_K = A_K^{-1}N_{K+1}^{-(q+\beta)} \to 0.
\]
Consequently, for every $\delta>0$,
\begin{equation}\label{positive-time-energy-subcritical}
\sup_K\sup_{t\geq\delta}\|X^K(t)\|_{\ell^2}<\infty.
\end{equation}
Indeed, \eqref{H-norm-K} bounds $\|X^K(\delta)\|_{\ell^2}$ uniformly in $K$, and
the Galerkin energy is non-increasing.

The uniform bound \eqref{critical-uniform-bound} also bounds every fixed
component and its derivative uniformly in $K$. Using the Arzel\`a--Ascoli theorem and passing to a diagonal subsequence, still denoted by $X^K$, we obtain a solution $X(t)$ of \eqref{Obukhov} such that
\[
X_k^K\to X_k \quad\text{locally uniformly on }[0,\infty)
\]
for every $k$.

Thanks to  \eqref{critical-uniform-bound}, for any $k$ and
$K\geq k+1$,
\[
0\leq X_k^K(t) \leq N_k^\alpha\int_0^t(X_{k+1}^K(s))^2\,ds
\leq B^2N_k^\alpha N_{k+1}^{2\beta}t.
\]
Hence
\[
\lim_{t\to0+}X_k(t)=0, \qquad k\geq1.
\]
Moreover, we can pass to the limit as $K\to \infty$ in \eqref{critical-uniform-bound} to obtain
$X_k(t)\leq BN_k^\beta$. Therefore, for every $s<-\beta=\alpha-2$, we have
\[
\lim_{t\to0+}\|X(t)\|_{H^s}=0.
\]

Since the flux in the energy balance is negative, we obtain, for $L\geq 2$
\[
\frac{d}{dt}\sum_{k=L}^{K+1}(X_k^K)^2
\leq-2\nu N_L^2\sum_{k=L}^{K+1}(X_k^K)^2.
\]
Fix $0<\delta<T$ and take $t_0=\delta/2$. Then, for $t\in[\delta,T]$,
\[
\sum_{k=L}^{K+1}(X_k^K(t))^2 \leq e^{-2\nu N_L^2(t-t_0)}\|X^K(t_0)\|_{\ell^2}^2 \lesssim_\delta e^{-\nu\delta N_L^2},
\]
where the last estimate follows from
\eqref{positive-time-energy-subcritical}.
Then for each mode
\[
X_k^K(t)^2 \lesssim_\delta e^{-\nu\delta N_k^2}, \qquad \forall t\in[\delta,T],
\]
and passing to the limit as $K \to \infty$ we get the same exponential decaying bound for $X_k(t)$. Hence 
\[
X\in C((0,\infty);H^s) \qquad\text{for every }s\in\mathbb R.
\]

Let
\[
I_k=[c_0N_k^{-2},C_0N_k^{-(q+\beta)}].
\]
Since $q<\alpha$, one has $q+\beta<2$, so these intervals are nonempty after
increasing $C_0$, if necessary. For every fixed $k$, by uniform convergence
on $I_k$ and \eqref{critical-flux-bound}, we obtain
\[
\max_{t\in I_k}N_k^{-\beta}X_k(t) =
\lim_{K\to\infty}\max_{t\in I_k}N_k^{-\beta}X_k^K(t)
\geq m_0.
\]
Choose $t_k\in I_k$ with
$X_k(t_k)\geq m_0N_k^\beta$. Since $q+\beta>0$, $t_k\to0$. Moreover, for
every $r>\alpha-2=-\beta$,
\[
\begin{split}
t_k^{\frac{r+\beta}{2}}\|X(t_k)\|_{H^r}
&\geq m_0t_k^{\frac{r+\beta}{2}}N_k^{r+\beta}\\
&\geq m_0c_0^{\frac{r+\beta}{2}}.
\end{split}
\]
This proves \eqref{blow-up_rate}.
\end{proof}

\begin{remark}
The amplitudes $A_K$ are selected by a finite-dimensional shooting argument;
a fixed coefficient $A_*$ is not needed. The condition
$\lambda^{2-\alpha}>2$ enters only in
Lemma~\ref{le-one-shell-transfer}. Removing this separation condition would
require a sharper transfer estimate.
\end{remark}

\subsection{The energy-critical case}
In this case $\alpha=2$, we take a similar route as in Subsection \ref{ss:subcritical} for the energy-subcritical situation. 

Fix $\sigma>0$ as in \eqref{eq:sigma-separation}, and set
\begin{equation}\label{eq:critical-weighted-constants}
a:=\lambda^{2\sigma}, \qquad
b:=\lambda^{-(2+\sigma)}, \qquad
\Lambda:=\lambda^{1+\sigma/2}.
\end{equation}
For a Galerkin solution $X_k^{K,A}$ to \eqref{eq:galerkin} with
\[
F_K=A_K N_{K+1}^{\sigma}, \quad \mbox{for some} \quad A_k>0,
\]
define
\[
Y_k^{K,A}:=N_k^{-\sigma}X_k^{K,A}.
\]
It follows from \eqref{eq:galerkin} that
\begin{equation}\label{eq:critical-normalized-system}
(Y_k^{K,A})'
=N_k^{2+\sigma}
\left(-\nu N_k^{-\sigma}Y_k^{K,A}+a(Y_{k+1}^{K,A})^2-bY_{k-1}^{K,A}Y_k^{K,A}\right),
\qquad 1\leq k\leq K+1,
\end{equation}
with $Y_0^{K,A}=Y_{K+2}^{K,A}=0$. 

\begin{lemma}
\label{lem:critical-one-shell-transfer}
There exists $B>0$, depending only on $\nu,\lambda,\sigma$, with the
following property.  Let $Y$ be a non-negative solution of
\eqref{eq:critical-normalized-system}.  If, for some $1\leq j\leq K$ and
$t_0\geq0$,
\[
Y_{j+1}(t_0)=B, \qquad Y_{j-1}(t)<B\quad \forall t\geq t_0,
\]
where $Y_0\equiv0$ when $j=1$, then
\[
Y_j(t)= B \qquad \text{for some} \quad t\geq t_0.
\]
\end{lemma}
\begin{proof}
Denote
\begin{equation}\label{eq:critical-eta}
\eta_\Lambda
:=(2\Lambda^2)^{-\frac{2\Lambda^2}{2\Lambda^2-1}}.
\end{equation}
Since $\Lambda>1$,
\[
(2\Lambda^2)^{1/(2\Lambda^2-1)}<2, \qquad
\eta_\Lambda>\frac{1}{4\Lambda^2}.
\]
Therefore \eqref{eq:sigma-separation} implies
\begin{equation}\label{eq:critical-gain-condition}
a\eta_\Lambda>b.
\end{equation}
Indeed,
\[
a\eta_\Lambda
>\frac{\lambda^{2\sigma}}{4\lambda^{2+\sigma}}
=\frac{\lambda^{2\sigma}}4\,b
\geq b.
\]

Choose $B$ so large that
\begin{equation}\label{eq:critical-choice-B}
(a\eta_\Lambda-b)B>\nu N_1^{-\sigma},
\end{equation}
and set
\[
d:=\nu N_1^{-\sigma}+bB.
\]
Then the conclusion of the lemma follows from a similar contradiction argument as in the proof of Lemma \ref{le-one-shell-transfer}.

\end{proof}

\begin{proposition}
\label{prop:critical-shooting}
There are amplitudes $A_K\in(0,B)$ such that the solutions of
\eqref{eq:galerkin} with
\begin{equation}\label{eq:critical-shooting-data}
F_K=A_KN_{K+1}^{\sigma}
\end{equation}
satisfy
\begin{equation}\label{eq:critical-weighted-barrier}
0\leq N_k^{-\sigma}X_k^K(t)\leq B
\qquad \mbox{for} \quad 1\leq k\leq K+1,\quad t\geq0
\end{equation}
and
\begin{equation}\label{eq:critical-first-shell-normalization}
\max_{t\geq0}N_1^{-\sigma}X_1^K(t)=\frac B2.
\end{equation}
\end{proposition}
\begin{proof}
From the energy balance we get
\[
\|X^{K,A}(t)\|_{\ell^2} \leq AN_{K+1}^{\sigma}e^{-\nu N_1^2t}.
\]
Thus each $Y_k^{K,A}$ tends to zero as $t\to\infty$.  For fixed $K$,
using finite-time continuous dependence on $A$, together with the above, we obtain that $A\mapsto Y_1^{K,A}$ is
continuous into $C_0([0,\infty))$, and hence
\[
\Phi_K(A):=\max_{t\geq0}Y_1^{K,A}(t)
\]
is continuous on $[0,B]$. Thus the maximum is attained.

We claim that
\begin{equation}\label{eq:critical-downward-barrier}
\Phi_K(A)<B \quad \Longrightarrow\quad
\max_{t\geq0}Y_j^{K,A}(t)<B \quad\text{for every }1\leq j\leq K+1.
\end{equation}
Assume the contrary, and let $m$ be the smallest index whose component reaches $B$:
\[
m := \min \{ j: \max_{t\geq 0}Y_j^{K,A}(t) \geq B \}.
\]
Then $m\geq2$, and the component $Y_{m-2}^{K,A}$ remains below $B$ for all
time.  At the first time when $Y_m^{K,A}=B$, Lemma~\ref{lem:critical-one-shell-transfer}, applied with $j=m-1$, implies that
$Y_{m-1}^{K,A}$ reaches $B$ at some time as well, contradicting the minimality of $m$.

For $A=B$, the top shell is nonincreasing and starts
at $B$, so its maximum is equal to $B$.  Hence
\eqref{eq:critical-downward-barrier} implies $\Phi_K(B)\geq B$. Note  $\Phi_K(0)=0$. By continuity, there is $A_K\in(0,B)$ such that
$\Phi_K(A_K)=B/2$.  Applying \eqref{eq:critical-downward-barrier} to this
amplitude gives \eqref{eq:critical-weighted-barrier}, while the choice of
$A_K$ gives \eqref{eq:critical-first-shell-normalization}.

\end{proof}

\begin{proof}[Proof of Theorem~\ref{th:critical-blowup}]
Let $X^K$ be the solutions obtained in Proposition~\ref{prop:critical-shooting}. The weighted barrier implies
\begin{equation}\label{eq:critical-X-barrier}
0\leq X_k^K(t)\leq BN_k^\sigma.
\end{equation}
In particular, $X_1^K\leq BN_1^\sigma$ globally. Lemma \ref{le-Lyapunov} together with
\[
H_{1,K}(0)=A_K^2N_{K+1}^{2+2\sigma}>0
\]
yields
\begin{equation}\label{eq:critical-positive-time-H1}
\sum_{k=1}^{K+1}N_k^2(X_k^K(t))^2\lesssim t^{-2}, \qquad t>0,
\end{equation}
up to a constant independent of $K$.

\smallskip
\noindent\emph{Compactness and nontriviality.}
Let $t_K$ be the first time at which $Y_1^K(t_K)=B/2$.  Since
\[
(Y_1^K)'\leq aN_1^{2+\sigma}B^2,
\]
we have
\[
t_K\geq \frac{1}{2aBN_1^{2+\sigma}}.
\]
On the other hand, coercivity gives
$H_{1,K}(t_K)\gtrsim B^2N_1^{2+2\sigma}$, while \eqref{Lyapunov-decay} implies $H_{1,K}(t_K)\lesssim t_K^{-2}$.
Thus $t_K\leq C$, uniformly in $K$.  After passing to a subsequence,
\begin{equation}\label{eq:critical-tK-limit}
t_K\to t_*>0.
\end{equation}

For each fixed $k$, thanks to \eqref{eq:critical-X-barrier}, we have that
$X_k^K$ and $(X_k^K)'$ are uniformly bounded on $[0,\infty)$ for all sufficiently
large $K$. By a diagonal
Arzel\`a--Ascoli argument, we obtain a subsequence, still denoted by
$X^K$, and non-negative functions $X_k\in C([0,\infty))$ such that
\begin{equation}\label{eq:critical-component-convergence}
X_k^K\to X_k \quad\text{locally uniformly on }[0,\infty)
\end{equation}
for every fixed $k$.  Passing to the limit in the integral form of the
Galerkin equations shows that, with $X_0\equiv0$, each $X_k$ belongs to
$C^1([0,\infty))$ and $X=(X_k)_{k\geq1}$ solves
\eqref{Obukhov} componentwise.  Moreover,
\eqref{eq:critical-tK-limit} and
\eqref{eq:critical-first-shell-normalization} imply
\begin{equation}\label{eq:critical-nontriviality}
X_1(t_*)=\frac B2N_1^\sigma>0.
\end{equation}

\smallskip
\noindent\emph{Positive-time regularity and the zero trace.} 
Fix $0<\delta<T$.  From \eqref{eq:critical-positive-time-H1},
\[
\sup_K\|X^K(\delta/2)\|_{\ell^2}^2\lesssim_\delta1.
\]
From the energy balance we get, for $L\geq1$,
\[
\frac{d}{dt}\sum_{k=L}^{K+1}(X_k^K)^2
\leq-2\nu N_L^2\sum_{k=L}^{K+1}(X_k^K)^2.
\]
Hence, for $t\in[\delta,T]$,
\begin{equation}\label{eq:critical-positive-time-tail}
X_L^K(t)^2 \leq \sum_{k=L}^{K+1}(X_k^K(t))^2 \lesssim_\delta e^{-\nu\delta N_L^2}.
\end{equation}
Passing to the limit as $K \to \infty$, this exponentially decaying bound implies
\begin{equation}\label{eq:critical-positive-time-strong}
X^K\to  X \qquad\text{in} \quad C([\delta,T];H^r),
\end{equation}
for every $r\in\mathbb R$. Thus $X\in C((0,\infty);H^r)$ for every $r\in\mathbb R$.

Passing to the limit in the weighted barrier, we get
\begin{equation}\label{eq:critical-limit-barrier}
0\leq X_k(t)\leq BN_k^\sigma.
\end{equation}
Since the negative terms in \eqref{Obukhov} may be discarded,
\begin{equation}\label{eq:critical-pointwise-small-time}
0\leq X_k(t)
\leq N_k^2\int_0^tX_{k+1}(s)^2\,ds
\leq B^2\lambda^{2\sigma}tN_k^{2+2\sigma}.
\end{equation}
Thus $X_k(t)\to0$ as $t\to0+$ for every fixed $k$.  If $s<-\sigma$, then
\[
N_k^{2s}X_k(t)^2\leq B^2N_k^{2(s+\sigma)},
\]
and the right-hand side is summable, which implies that
$\|X(t)\|_{H^s}\to 0$ as $t\to0+$.  Defining $X(0)=0$, we conclude that
\[X\in C([0,\infty);H^s) \qquad \mbox{for} \quad s<-\sigma.
\]

\smallskip
\noindent\emph{Blow-up of the energy.}
For $0<\tau<t$, the energy balance in the first $L$ shells reads
\[
\frac12\sum_{k=1}^L X_k(t)^2
+\nu\int_\tau^t\sum_{k=1}^L N_k^2X_k(s)^2\,ds =
\frac12\sum_{k=1}^L X_k(\tau)^2
+\int_\tau^tN_L^2X_LX_{L+1}^2\,ds.
\]
Since $X\in C([\tau,t];H^1)$,
\[
N_L^2X_LX_{L+1}^2
\leq\lambda^{-2}N_L^{-1}\|X\|_{H^1}^3,
\]
so the last integral tends to zero as $L\to\infty$.  Therefore
\begin{equation}\label{eq:critical-limit-energy}
\frac12\|X(t)\|_{\ell^2}^2+\nu\int_\tau^t\|X(s)\|_{H^1}^2\,ds
=\frac12\|X(\tau)\|_{\ell^2}^2, \qquad 0<\tau<t.
\end{equation}
In particular, $\|X(t)\|_{\ell^2}$ is nonincreasing in $t$, and its limit
as $t\to0+$ exists in $(0,\infty]$ by
\eqref{eq:critical-nontriviality}.

Suppose that this limit were finite, and let
\[
M:=\sup_{0<t\leq t_*}\|X(t)\|_{\ell^2}<\infty.
\]
Letting $\tau\to0+$ in \eqref{eq:critical-limit-energy} shows that
\begin{equation}\label{eq:critical-finite-dissipation}
\int_0^{t_*}\|X(s)\|_{H^1}^2\,ds<\infty.
\end{equation}
Since $X_k(0)=0$, the energy balance in the first $L$ shells on $[0,t]$ is
\begin{equation}\label{eq:critical-low-mode-zero}
\frac12\sum_{k=1}^{L}X_k(t)^2
+\nu\int_0^t\sum_{k=1}^{L}N_k^2X_k(s)^2\,ds
=\int_0^tN_L^2X_L(s)X_{L+1}(s)^2\,ds.
\end{equation}
For $t\leq t_*$, the right-hand side is bounded by
\[
\lambda^{-2}M\int_0^tN_{L+1}^2X_{L+1}(s)^2\,ds,
\]
which tends to zero as $L\to\infty$ by
\eqref{eq:critical-finite-dissipation}.  Passing to the limit in
\eqref{eq:critical-low-mode-zero}, we obtain
\[
\frac12\|X(t)\|_{\ell^2}^2 + \nu\int_0^t\|X(s)\|_{H^1}^2\,ds=0,
\]
contradicting \eqref{eq:critical-nontriviality}.  Hence
\eqref{eq:critical-energy-blowup} holds.

\smallskip
\noindent\emph{The logarithmic lower bound.}
Set $m_*:=X_1(t_*)>0$.  From
\eqref{eq:critical-pointwise-small-time},
\[
\sum_{k=1}^{K}X_k(t)^2
\lesssim t^2N_K^{4+4\sigma}.
\]
Choose $\theta>0$ sufficiently small.  For all sufficiently small $t$, let
$K(t)$ be the largest integer such that
\begin{equation}\label{eq:critical-Kt}
tN_{K(t)}^{2+2\sigma}\leq\theta.
\end{equation}
Then
\begin{equation}\label{eq:critical-low-energy-small}
\sum_{k=1}^{K(t)}X_k(t)^2\leq\frac{m_*^2}{4}.
\end{equation}
For every $1\leq L\leq K(t)$, the low-mode balance between $t$ and $t_*$
gives
\begin{equation}\label{eq:critical-flux-lower}
\int_t^{t_*}N_L^2X_L(s)X_{L+1}(s)^2\,ds
\geq\frac{3m_*^2}{8}.
\end{equation}
Summing over $1\leq L\leq K(t)$ and using
\[
\sum_{L=1}^{K(t)}N_L^2X_LX_{L+1}^2
=\lambda^{-2}\sum_{L=1}^{K(t)}X_LN_{L+1}^2X_{L+1}^2
\leq\lambda^{-2}\|X(s)\|_{\ell^2}\|X(s)\|_{H^1}^2,
\]
we obtain
\[
K(t) \lesssim \|X(t)\|_{\ell^2} \int_t^{t_*}\|X(s)\|_{H^1}^2\,ds
\lesssim\|X(t)\|_{\ell^2}^3.
\]
Here we used the monotonicity of $\|X(s)\|_{\ell^2}$ and
\eqref{eq:critical-limit-energy}.  The maximality in
\eqref{eq:critical-Kt} implies that, after decreasing $t_0>0$ if necessary,
\[
K(t)\gtrsim1+\log\frac{t_0}{t}, \qquad 0<t<t_0,
\]
which implies \eqref{eq:log-lower-bound}.

\end{proof}

\section{Second approach: an explicit multiscale construction}
\label{sec:second_approach}

The purpose of this section is to outline a second approach toward
Theorems~\ref{th:alpha<2} and \ref{th:alpha>2} which is more amenable to
proving analogous theorems for the full Navier--Stokes equations.  For
brevity, we specialize to the case $\alpha>2$, as the $\alpha<2$ strategy
closely aligns with the main construction in~\cite{cheskidov2025instantaneous}.
Throughout the section, $N_k=\lambda^k$ and $\alpha>2$.

Unlike the argument in Section~\ref{sec:first_approach}, the scheme below
uses no monotone quantity adapted to the model.  It consists of an explicit
construction with a perturbative corrector, analogous to successful arguments in the PDE setting such as \cite{MR5008166} and \cite{cheskidov2025instantaneous}.

We aim at a solution of \eqref{Obukhov} with
\begin{equation}\label{second-approach-goal}
\lim_{t\to0+}X_k(t)=0,
\qquad k=1,2,\dots,
\quad\text{and}\quad
\liminf_{t\to0+}\|X(t)\|_{\ell^2}>0.
\end{equation}
This already yields the failure of uniqueness; the exact normalization
$\lim_{t\to0+}\|X(t)\|_{\ell^2}=1$ obtained in Theorem~\ref{th:alpha>2} can be achieved by a
careful tuning of amplitudes.

The solution admits the decomposition
\[
X=X^p+X^c,
\]
with $X^p$ the ``principal part'' arising from a fairly explicit
construction and $X^c$ a corrector, through the following steps.
Here $\delta>0$ is a small parameter of the construction and $T>0$ is small.
\begin{itemize}[left=3em]
\item[Step 1.] Construct an explicit $X^p$ with $0\leq X_k^p\leq1$ such that
$X_k^p(t)\to0$ as $t\to0+$ for every fixed $k$, and
\[
\bigl|\|X^p(t)\|_{\ell^2}^2-1\bigr|\lesssim\delta
\qquad\text{for all }0<t\leq T.
\]
\item[Step 2.] Show that $X^p$ solves \eqref{Obukhov} up to an error $f$ which
is small both componentwise in $L^1_t$ and in the scale-invariant norm
$\sup_{0<t\leq T}t\|f(t)\|_{\ell^2}$.
\item[Step 3.] Solve the corrector equation with $X^c(0)=0$, obtaining a
solution with
\[
\sup_{0<t\leq T}\|X^c(t)\|_{\ell^2}\leq\tfrac12 .
\]
\end{itemize}
Steps 1 and 2 are carried out below, while Step~3 is omitted for brevity.\footnote{In the PDE setting, this step can be quite delicate and typically requires extra input beyond critical smallness of the force. In the Obukhov setting, the correction would be facilitated by the absence of certain non-local interactions in frequency space.} We begin with some heuristic comments.

\subsection{The two-shell mechanism}
\label{ss:pulse}

The central mechanism is a series of energy transfers. One such transfer, from
shell $k+1$ to shell $k$, is described by the system
\begin{equation}\label{Principal3}
\begin{split}
\frac{d}{dt}\bar X_{k+1} &= - \nu N_{k+1}^2\bar X_{k+1}  - N_{k}^\alpha  X_{k}^p \bar X_{k+1},\\
\frac{d}{dt}X_k^p &= - \nu N_k^2X_k^p + N_k^\alpha \bar X_{k+1}^2 - N_{k-1}^\alpha \bar X_{k-1} X_k^p,\\
\frac{d}{dt}\bar X_{k-1} &= - \nu N_{k-1}^2\bar X_{k-1} + N_{k-1}^\alpha (X_{k}^p)^2,
\end{split}
\end{equation}
obtained from \eqref{Obukhov} by retaining only the interactions among
three consecutive shells.  A direct computation gives
\begin{equation}\label{triple-energy}
\frac{d}{dt}
\left(
\bar X_{k-1}^2+(X_k^p)^2+\bar X_{k+1}^2
\right)
=
-2\nu
\left(
N_{k-1}^2\bar X_{k-1}^2
+N_k^2(X_k^p)^2
+N_{k+1}^2\bar X_{k+1}^2
\right),
\end{equation}
so that, up to dissipation, the triple simply moves a fixed amount of
energy downward. 

\begin{remark}\label{rem:diss-perturbative}
In order to produce an almost complete transfer of energy at each step in the
cascade, the principal part must be arranged in such a way that the
Laplacian term is perturbative. In this way, the higher mode can transfer all of its energy via the negative nonlinear term before the dissipation can
have a significant effect. Writing $t_k$ for the duration of the transfer into shell $k$, this suggests the constraint
\begin{equation}\label{diss-perturbative}
t_k\ll(\nu N_{k+1}^2)^{-1}.
\end{equation}
This forces the construction into the supercritical regime $\alpha>2$.
\end{remark}

Assuming \eqref{diss-perturbative} for the purpose of discussion, we may drop the dissipative terms in
\eqref{Principal3} during the transfer window. Similarly,
$\bar X_{k-1}$ has not yet been excited so its contribution to the nonlinearity can be neglected on this time scale. What remains is the
two-shell system
\begin{equation}\label{Principal2}
\begin{split}
\frac{d}{dt}\bar X_{k+1} &=   - N_{k}^\alpha  X_{k}^p \bar X_{k+1},\\
\frac{d}{dt}X_k^p &=  N_k^\alpha \bar X_{k+1}^2
\end{split}
\end{equation}
with initial data
\[
\bar X_{k+1}(0) = a_{k+1}, \qquad X^p_{k}(0) = 0 .
\]
This system has an exact solution,
\begin{equation}\label{two-shell-orbit}
\begin{split}
X_k^p(t)&=a_{k+1}\tanh(N_k^\alpha a_{k+1}t)\\
\bar X_{k+1}(t)&=a_{k+1}\sech(N_k^\alpha a_{k+1}t).
\end{split}
\end{equation}
The trajectory lies on a circle of radius $a_{k+1}$, consistent with energy
conservation for the pair of modes; thus, at leading order, $a_k=a_{k+1}$,
as expected.  The energy transfer occurs on the time scale
$(N_k^\alpha a_{k+1})^{-1}$, which by Remark~\ref{rem:diss-perturbative}
outpaces dissipation as long as
\begin{equation}\label{dissipation_is_negligible_constraint}
a_{k+1}\gg \nu\,N_k^{-\alpha}N_{k+1}^{2}.
\end{equation}
For the geometric choice $N_k=\lambda^k$ and amplitudes $a_k\sim1$, the
right-hand side of \eqref{dissipation_is_negligible_constraint} reduces to
$\nu\lambda^2N_k^{2-\alpha}$ which tends to $0$ along the shells as long as $\alpha>2$.  This is the same competition between the nonlinear
transfer time and the dissipative time scale that the Lyapunov argument of
Section~\ref{sec:first_approach} exploits abstractly.

Normalizing $a_k\equiv1$, \eqref{two-shell-orbit} says that shell $k$ is
switched on at time $t\sim N_k^{-\alpha}$ and, one step later, switched off
at time $t\sim N_{k-1}^{-\alpha}$ as its energy passes to shell $k-1$.
Since $\sum_kN_k^{-\alpha}<\infty$, the whole cascade can be completed in
finite time, entering from infinite frequency at $t=0$. The resulting dynamics are shown in Figures~\ref{fig:cascade} and~\ref{fig:cascade-numerical}.

\begin{figure}[ht]
\centering
\begin{tikzpicture}
    \draw[thick] (-0.5,0) -- (11.2,0);

    % Time 0
    \draw[thick] (-0.5,0.1) -- (-0.5,-0.1);
    \node[below] at (-0.5,-0.2) {$0$};
    \node at (0.25,-0.45) {$\cdots$};

    % Ticks
    \draw[thick] (1,0.1) -- (1,-0.1);
    \node[below] at (1,-0.2) {\small $N_{k+1}^{-\alpha}$};
    \draw[thick] (3.25,0.1) -- (3.25,-0.1);
    \node[below] at (3.25,-0.2) {\small $T_{k+1}$};
    \draw[thick] (5.5,0.1) -- (5.5,-0.1);
    \node[below] at (5.5,-0.2) {\small $N_{k}^{-\alpha}$};
    \draw[thick] (7.75,0.1) -- (7.75,-0.1);
    \node[below] at (7.75,-0.2) {\small $T_{k}$};
    \draw[thick] (10,0.1) -- (10,-0.1);
    \node[below] at (10,-0.2) {\small $N_{k-1}^{-\alpha}$};
    \node at (10.75,-0.45) {$\cdots$};
    \draw[decorate,decoration={brace,amplitude=5pt,mirror}] (1,-1.05) -- (5.45,-1.05);
    \node[below] at (3.25,-1.3) {\small $X^p_{k+1}\approx1$};
    \draw[decorate,decoration={brace,amplitude=5pt,mirror}] (5.55,-1.05) -- (10,-1.05);
    \node[below] at (7.75,-1.3) {\small $X^p_{k}\approx1$};

    % Transfers
    \draw[decorate,decoration={brace,amplitude=3pt}] (0.35,0.3) -- (1.65,0.3);
    \node[above] at (1,0.5) {\small $X^p_{k+2}\rightsquigarrow X^p_{k+1}$};
    \draw[decorate,decoration={brace,amplitude=3pt}] (4.85,0.3) -- (6.15,0.3);
    \node[above] at (5.5,0.5) {\small $X^p_{k+1}\rightsquigarrow X^p_{k}$};
\end{tikzpicture}
\caption{The inverse cascade carried by the principal part.  Shell $k$ is
excited on the time scale $N_k^{-\alpha}$ and relinquishes its energy to
shell $k-1$ on the time scale $N_{k-1}^{-\alpha}$. The intermediate times $T_k$, at which the
ansatz \eqref{Simple_PP} switches between its two branches, satisfy
$N_k^{-\alpha}\ll T_k\ll N_{k-1}^{-\alpha}$.}
\label{fig:cascade}
\end{figure}
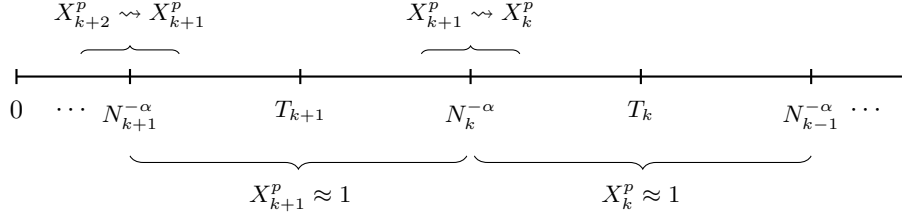

\begin{figure}[ht]
\centering
\includegraphics{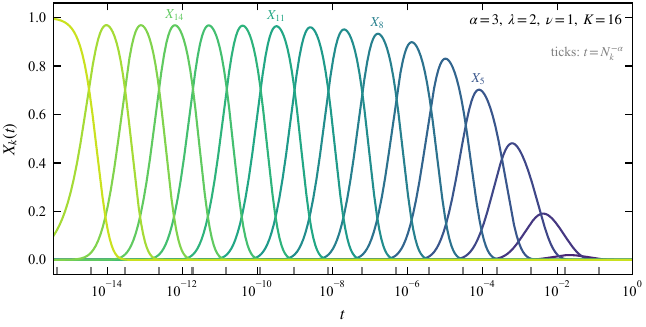}
\caption{The inverse cascade of Theorem~\ref{th:alpha>2}, computed from the
Galerkin system \eqref{eq:galerkin} with $\alpha=3$, $\lambda=2$, $\nu=1$,
$K=16$ and $F_K=1$. Each curve is one shell $X_k$. Boundedly-many shells are significantly activated at a given time, and shell $k$ switches on at
$t\approx2.7\,N_k^{-\alpha}$.}
\label{fig:cascade-numerical}
\end{figure}

\subsection{The principal part}
\label{ss:principal-part}

If the scales are sufficiently separated, the pairs
$(X_k^p,\bar X_{k+1})$, $(X_{k+1}^p,\bar X_{k+2})$, etc.\ decouple, and the
principal part can be defined very simply by gluing together the two
branches of \eqref{two-shell-orbit}.  We therefore define
\begin{equation}\label{Simple_PP}
X_k^p=\chi_k^1\tanh(N_k^\alpha t)+\chi_k^2\sech(N_{k-1}^\alpha t),
\qquad k\geq1,
\end{equation}
where the times $T_k$ satisfy
\begin{equation}\label{cutoff_time_constraint}
N_k^{-\alpha}\ll T_k \ll N_{k-1}^{-\alpha},
\end{equation}
and the cutoffs are chosen so that
\[
\begin{gathered}
0\leq \chi_k^1,\chi_k^2\leq1,
\qquad
\chi_k^2=1-\chi_k^1,
\qquad
\|\partial_t\chi_k^i\|_\infty\lesssim T_k^{-1},\quad i=1,2,\\
\chi_k^1=1 \text{ on } [0,T_k],
\qquad
\supp\chi_k^1\subset[0,2T_k].
\end{gathered}
\]
In particular, $0\leq X^p_k\leq1$. The identities used below,
\begin{equation}\label{cutoff-identities}
\chi_{k-1}^1\chi_k^1\equiv\chi_k^1,\quad
\chi_{k-1}^2\chi_k^1\equiv0,\quad
\chi_{k-1}^2\chi_k^2\equiv\chi_{k-1}^2,
\end{equation}
hold as soon as the cutoffs are nested, that is, $2T_k\leq T_{k-1}$ for the
shells under consideration.

For the residual estimate of Subsection~\ref{ss:error} to be small, we
impose the following quantitative scale separation conditions: for some
sufficiently large $k_0$,
\begin{equation}\label{delta-star}
\delta_*
:=
\sup_{k\geq k_0}
\left(
N_k^\alpha T_{k+1}
+
(N_{k-1}^\alpha T_k)^2
+
\nu N_k^2N_{k-1}^{-\alpha}
\right)
\ll1,
\end{equation}
and
\begin{equation}\label{eta-star}
\eta_*
:=
\sup_{k\geq k_0}
\left(
e^{-N_k^\alpha T_k}
+
e^{-N_{k-1}^\alpha T_{k-1}}
\right)
\ll1 .
\end{equation}
For $N_k=\lambda^k$, one may take
\begin{equation}\label{T-choice}
T_k=\delta N_{k-1}^{-\alpha},
\qquad
\lambda^{-\alpha}\ll\delta\ll1,
\end{equation}
with the shell ratio large enough to leave room for such a $\delta$; note
that \eqref{T-choice} forces $\lambda^\alpha\geq2$, so the cutoffs are
automatically nested and \eqref{cutoff-identities} holds.  This choice gives
\[
N_k^\alpha T_{k+1}=\delta,
\qquad
(N_{k-1}^\alpha T_k)^2=\delta^2,
\qquad
e^{-N_k^\alpha T_k}=e^{-N_{k-1}^\alpha T_{k-1}}=e^{-\delta\lambda^\alpha},
\]
so that $\eta_*\ll1$ and the first two terms in $\delta_*$ are $O(\delta)$.
Choosing $k_0$ so large that
\[
\nu N_k^2N_{k-1}^{-\alpha}=\nu\lambda^\alpha N_k^{2-\alpha}\ll1
\qquad\text{for }k\geq k_0
\]
(using that $\alpha>2$), then gives
$\delta_*,\eta_*\ll1$.  For $1\leq k<k_0$, we lose scale separation and the ansatz \eqref{Simple_PP} is
not justified; instead, these modes are allowed to evolve by
\eqref{Obukhov}, with their exact behavior not affecting the desired conclusions
\eqref{second-approach-goal} near $t=0$.

We now verify Step 1 for the ansatz \eqref{Simple_PP}.

\begin{lemma}\label{lem:principal-properties}
Assume \eqref{T-choice}.  Then $0\leq X^p_k\leq1$ for all $k$ and $t$, and:
\begin{itemize}[left=1em]
\item[(i)] for every fixed $k\geq k_0$ one has
$X^p_k(t)=\tanh(N_k^\alpha t)$ for $0\leq t\leq T_k$, and in particular
$X^p_k(t)\to0$ as $t\to0+$;
\item[(ii)] $X^p_k(N_k^{-\alpha})>\tfrac12$;
\item[(iii)] there is $c>0$ such that
\[
\bigl|\|X^p(t)\|_{\ell^2}^2-1\bigr|
\lesssim
\delta^2+e^{-c\delta\lambda^\alpha},
\qquad 0<t\leq T_{k_0}.
\]
\end{itemize}
\end{lemma}

\begin{proof}
The claim (i) is immediate from $\chi^1_k=1$ on $[0,T_k]$, and (ii) follows since
$N_k^{-\alpha}\leq T_k$ by \eqref{T-choice}.

For (iii), fix $0<t\leq T_{k_0}$ and let $m\geq k_0$ be the unique index
with $T_{m+1}\leq t<T_m$; set $\theta:=N_m^\alpha t$, so that
$\theta\in[\delta,\delta\lambda^\alpha)$ by \eqref{T-choice}.

Suppose first that $t\geq2T_{m+1}$.  For $j\leq m$ we have $t\leq T_m\leq T_j$,
hence $X^p_j=\tanh(N_j^\alpha t)$; for $j\geq m+1$ we have $t\geq2T_{m+1}\geq2T_j$,
hence $X^p_j=\sech(N_{j-1}^\alpha t)$.  Since $N_j^\alpha t=\theta\lambda^{\alpha(j-m)}$,
\[
\|X^p(t)\|_{\ell^2}^2
=
\sum_{i\geq0}\tanh^2\!\bigl(\theta\lambda^{-\alpha i}\bigr)
+
\sum_{i\geq0}\sech^2\!\bigl(\theta\lambda^{\alpha i}\bigr).
\]
The $i=0$ terms sum to $\tanh^2\theta+\sech^2\theta=1$.  In the first sum
the terms with $i\geq1$ obey
$\tanh^2(\theta\lambda^{-\alpha i})\leq(\theta\lambda^{-\alpha i})^2\lesssim\delta^2\lambda^{-2\alpha(i-1)}$,
using $\theta\lambda^{-\alpha}<\delta$; in the second, the terms with
$i\geq1$ obey $\sech^2(\theta\lambda^{\alpha i})\lesssim e^{-2\delta\lambda^{\alpha i}}$.
Both tails are therefore $O(\delta^2)+O(e^{-2\delta\lambda^\alpha})$.

If instead $T_{m+1}\leq t<2T_{m+1}$, then $N_{m+1}^\alpha t\geq\delta\lambda^\alpha\gg1$
and $N_m^\alpha t\in[\delta,2\delta)$, so
\[
X^p_{m+1}
=\chi^1_{m+1}\tanh(N_{m+1}^\alpha t)+\chi^2_{m+1}\sech(N_m^\alpha t)
=1+O(\delta^2)+O(e^{-\delta\lambda^\alpha}),
\]
because both branches are within $O(\delta^2)+O(e^{-\delta\lambda^\alpha})$
of $1$ there and $\chi^1_{m+1}+\chi^2_{m+1}=1$.  The remaining shells are
estimated exactly as before.  This proves (iii).
\end{proof}

Thus, at any small time, exactly one shell is excited, and the excited shell
carries all but $O(\delta^2)$ units of energy.

\subsection{The residual}
\label{ss:error}

Write $Q$ for the nonlinearity in \eqref{Obukhov},
\begin{equation}\label{Qdef}
Q(Y)_k:=N_k^\alpha Y_{k+1}^2-N_{k-1}^\alpha Y_{k-1}Y_k,
\qquad Y_0:=0,
\end{equation}
so that \eqref{Obukhov} reads $X_k'=-\nu N_k^2X_k+Q(X)_k$. Step 2 can be executed as follows.

\begin{proposition}\label{prop:residual}
Assume $\alpha>2$, $N_k=\lambda^k$ and \eqref{cutoff_time_constraint}, and
let $\delta_*,\eta_*$ be as in \eqref{delta-star}--\eqref{eta-star}.  For
every $k\geq k_0$, the principal part \eqref{Simple_PP} satisfies
\[
(X_k^p)'
=
-\nu N_k^2X_k^p
+
Q(X^p)_k
+
f_k ,
\qquad
X_0^p=0,
\]
where
\[
\|f_k\|_{L^1([0,\infty))}
\lesssim
N_k^\alpha T_{k+1}
+
(N_{k-1}^\alpha T_k)^2
+
\nu N_k^2N_{k-1}^{-\alpha}
+
e^{-N_k^\alpha T_k}
+
e^{-N_{k-1}^\alpha T_{k-1}}
\;\lesssim\;
\delta_*+\eta_* .
\]
\end{proposition}

\begin{proof}
Insert \eqref{Simple_PP} into $(X^p_k)'+\nu N_k^2X^p_k-Q(X^p)_k$ and use the
cutoff identities \eqref{cutoff-identities} together with
$(\chi^1_k)'=-(\chi^2_k)'$.  Grouping the resulting terms according to their
origin gives
\[
f_k=f_k^{interact}+f_k^{diss}+f_k^{cutoff}+f_k^{NL_1}+f_k^{NL_2},
\]
where
\[
f_k^{interact}=N_k^\alpha(\chi_k^1-(\chi_{k+1}^2)^2)\sech^2(N_k^\alpha t)-N_{k-1}^\alpha\chi_{k-1}^2\sech(N_{k-1}^\alpha t)\tanh(N_{k-1}^\alpha t)
\]
collects the terms in which the time derivative hits the main part of
$X_k^p$ and which cancel against the ``important'' nonlinear interactions
appearing in \eqref{Principal2}; here the second term uses
$\chi^1_{k-1}\chi^2_k=\chi^2_k-\chi^2_{k-1}$, a consequence of
\eqref{cutoff-identities}.  Next,
\[
f_k^{diss}=\nu N_k^2\chi_k^1\tanh(N_k^\alpha t)+\nu N_k^2\chi_k^2\sech(N_{k-1}^\alpha t)
=\nu N_k^2X^p_k
\]
is the term from the Laplacian;
\[
f_k^{cutoff}=(\chi_k^1)'(\tanh(N_k^\alpha t)-\sech(N_{k-1}^\alpha t))
\]
appears when the time derivative hits the cutoff in \eqref{Simple_PP};
\[
f_k^{NL_1}=-N_k^\alpha(\chi_{k+1}^1)^2\tanh^2(N_{k+1}^\alpha t)-2N_k^\alpha\chi_{k+1}^1\chi_{k+1}^2\tanh(N_{k+1}^\alpha t)\sech(N_k^\alpha t)
\]
are the additional error terms from the first nonlinear term in
\eqref{Obukhov}; and
\[
f_k^{NL_2}=N_{k-1}^\alpha\chi_k^1\tanh(N_{k-1}^\alpha t)\tanh(N_k^\alpha t)+N_{k-1}^\alpha\chi_{k-1}^2\sech(N_{k-2}^\alpha t)\sech(N_{k-1}^\alpha t)
\]
are the additional error terms from the second nonlinear term in
\eqref{Obukhov}.

We now estimate the terms, using repeatedly the elementary bounds
\begin{equation}\label{elementary-bounds}
0\leq1-\tanh x\leq2e^{-2x},
\qquad
0\leq1-\sech x\leq \tfrac12x^2,
\qquad
\sech x\leq2e^{-x},
\qquad x\geq0 .
\end{equation}

The first term in $f_k^{interact}$ has a coefficient which vanishes on
$[2T_{k+1},T_k]$, while that of the second term is supported in
$[T_{k-1},\infty)$.  Then it is elementary to bound
\[
|f_k^{interact}(t)|\lesssim N_k^\alpha (\mathbf{1}_{[0,2T_{k+1}]}+e^{-N_k^\alpha t}\mathbf{1}_{[T_k,\infty)})+N_{k-1}^\alpha e^{-N_{k-1}^\alpha t}\mathbf{1}_{[T_{k-1},\infty)},
\]
from which we obtain
\[
\|f_k^{interact}\|_{L^1}\lesssim N_k^\alpha T_{k+1}+e^{-N_k^\alpha T_k}+e^{-N_{k-1}^\alpha T_{k-1}}.
\]

Next, using $\supp\chi_k^1\subset[0,2T_k]$, $\chi_k^2=0$ on $[0,T_k]$, and
$T_k\ll N_{k-1}^{-\alpha}$,
\begin{equation*}
\begin{split}
\|f_k^{diss}\|_{L^1}
&\lesssim
\nu N_k^2\left(T_k+
\int_{T_k}^\infty e^{-N_{k-1}^\alpha t}\,dt\right)\\
&\lesssim \nu N_k^2N_{k-1}^{-\alpha}.
\end{split}
\end{equation*}

Using that $\|\partial_t\chi_k^i\|_\infty\lesssim T_k^{-1}$ and
\eqref{elementary-bounds},
\begin{align*}
|f_k^{cutoff}(t)|&\lesssim T_k^{-1}(|\tanh(N_k^\alpha t)-1|+|\sech(N_{k-1}^\alpha t)-1|)\mathbf{1}_{[T_k,2T_k]}(t)\\
&\lesssim T_k^{-1}(e^{-2N_k^\alpha T_k}+(N_{k-1}^\alpha T_k)^2)\mathbf{1}_{[T_k,2T_k]}(t)
\end{align*}
and hence
\[
\|f_k^{cutoff}\|_{L^1}\lesssim e^{-2N_k^\alpha T_k}+(N_{k-1}^\alpha T_k)^2.
\]
Smallness of $f_k^{NL_1}$ is a result of the smaller coefficient:
\[
|f_k^{NL_1}(t)|\lesssim N_k^\alpha\mathbf{1}_{[0,2T_{k+1}]},
\qquad
\|f_k^{NL_1}\|_{L^1}\lesssim N_k^\alpha T_{k+1}.
\]
Finally, for $f_k^{NL_2}$ we find
\[
|f_k^{NL_2}(t)|\lesssim N_{k-1}^{2\alpha}t\, \mathbf{1}_{[0,2T_k]}(t)+N_{k-1}^\alpha e^{-N_{k-1}^\alpha t}\mathbf{1}_{[T_{k-1},\infty)}(t),
\]
\[
\|f_k^{NL_2}\|_{L^1}\lesssim (T_kN_{k-1}^\alpha)^2+e^{-N_{k-1}^\alpha T_{k-1}}.
\]
Summing the estimates concludes the proof.
\end{proof}

The componentwise bound of Proposition~\ref{prop:residual} is not by itself
enough to control the corrector, since it is not summable in $k$.  The same
pointwise estimates, however, yield a bound on the whole residual in the
scale-invariant norm suggested by Figure~\ref{fig:cascade}.

\begin{corollary}\label{cor:residual-scale-invariant}
Under the hypotheses of Proposition~\ref{prop:residual} and with the choice
\eqref{T-choice},
\begin{equation}\label{residual-scale-invariant}
\sup_{0<t\leq T_{k_0}}\ t\,\|f(t)\|_{\ell^2}
\ \lesssim\
\delta_*+\eta_*^{1/2}.
\end{equation}
\end{corollary}

\begin{proof}[Sketch]
Fix $t$ and let $m$ be as in the proof of
Lemma~\ref{lem:principal-properties}.  The indicator $\mathbf1_{[0,2T_{k+1}]}(t)$
is nonzero only for $N_k^\alpha\leq2\delta/t$, and for such $k$ the
corresponding coefficients $N_k^\alpha$ and $N_{k-1}^{2\alpha}t$ are
bounded by $2\delta/t$ and $4\delta^2/t$ respectively and decay
geometrically as $k$ decreases; this handles the first parts of
$f^{interact}$, $f^{NL_1}$ and $f^{NL_2}$.  On $[T_k,\infty)$ one has
$N_k^\alpha t\geq\delta\lambda^\alpha$, so
$N_k^\alpha e^{-N_k^\alpha t}\leq t^{-1}\sup_{s\geq\delta\lambda^\alpha}se^{-s}
\lesssim t^{-1}e^{-\delta\lambda^\alpha/2}=t^{-1}\eta_*^{1/2}$, and likewise
for the terms carrying $N_{k-1}^\alpha$.  The cutoff term is supported in
$[T_k,2T_k]$ for at most one value of $k$ (i.e., where $T_k\sim t$), and is
bounded there by $t^{-1}(\delta^2+\eta_*)$.  Finally,
$\|f^{diss}(t)\|_{\ell^2}=\nu\|X^p(t)\|_{H^2}\lesssim\nu N_m^2$ by
Lemma~\ref{lem:principal-properties}, and $t N_m^2\lesssim
N_m^2N_{m-1}^{-\alpha}\leq\nu^{-1}\delta_*$.
\end{proof}

\section{Inviscid blow-up}
\label{s:inviscid}
In the absence of viscosity, i.e. $\nu=0$, we have the inviscid Obukhov model
\begin{equation}\label{Obukhov-inviscid}
X_k'=N_k^\alpha X_{k+1}^2-N_{k-1}^\alpha X_{k-1}X_k,
\qquad k\geq1,
\qquad X_0=0.
\end{equation}
In order to construct an instantaneous blow-up, we need to keep in mind two important differences from the viscous construction. First, finite-dimensional Galerkin solutions conserve energy; second, the previous argument on comparing the nonlinear transfer time with the dissipative time scale is no longer applicable. 

To heuristically obtain the time scale of energy transfer in the inviscid case, we can consider the isolated two-shell system
\[
X_k'=N_k^\alpha X_{k+1}^2,
\qquad
X_{k+1}'=-N_k^\alpha X_kX_{k+1}.
\]
Starting from $X_k(t_0)=0$ and $X_{k+1}(t_0)=a>0$, its exact solution is
\begin{equation}\label{inviscid-two-shell-orbit}
\begin{split}
X_k(t)&=a\tanh\bigl(aN_k^\alpha(t-t_0)\bigr),\\
X_{k+1}(t)&=a\sech\bigl(aN_k^\alpha(t-t_0)\bigr).
\end{split}
\end{equation}
Thus energy is transferred from shell $k+1$ to shell $k$ on the time scale
$a^{-1}N_k^{-\alpha}$.  Since
$\sum_{j\geq k}N_j^{-\alpha}<\infty$ for every $\alpha>0$, one should expect that the inverse cascade can enter from infinity in finite time.

\begin{proof}[Proof of Theorem~\ref{th:inviscid}]
Set $\nu=0$ and $F_K=1$ in \eqref{eq:galerkin}.  The resulting
Galerkin solution satisfies
\begin{equation}\label{inviscid-Galerkin-energy}
\|X^K(t)\|_{\ell^2}=1
\qquad\text{for every }t\geq0.
\end{equation}
Fix $0<q<\alpha$ and define $H_{q,K}$ by \eqref{def-HqK}, with $c>0$
sufficiently small.  The norm equivalence \eqref{H-norm-K} is unchanged.
Repeating the cubic estimates in the proof of Lemma~\ref{le-Lyapunov}, but
omitting the viscous terms, we obtain
\[
-H_{q,K}'(t)
\geq
c_1\sum_{k=2}^{K+1}N_k^{\alpha+2q}(X_k^K(t))^3
-C_1(X_1^K(t))^3,
\]
where $c_1,C_1>0$ are independent of $K$.  By
\eqref{inviscid-Galerkin-energy}, $0\leq X_1^K\leq1$.  On the other hand,
\eqref{H-norm-2K} and \eqref{H-norm-K} imply
\[
H_{q,K}(t)^{3/2} \lesssim\sum_{k=1}^{K+1}N_k^{\alpha+2q}(X_k^K(t))^3
\lesssim 1+\sum_{k=2}^{K+1}N_k^{\alpha+2q}(X_k^K(t))^3.
\]
Consequently,
\begin{equation}\label{inviscid-Lyapunov-inequality}
H_{q,K}'(t) \leq - c_qH_{q,K}(t)^{3/2}+C_q,
\end{equation}
with constants independent of $K$.  Comparison with the scalar differential
inequality yields the uniform positive-time bound
\begin{equation}\label{inviscid-Lyapunov-bound}
H_{q,K}(t)\lesssim_q 1+t^{-2},
\qquad t>0.
\end{equation}
The additive constant in \eqref{inviscid-Lyapunov-inequality} is natural:
after the cascade reaches the low shells, the conserved energy remains there,
so one should not expect $H_{q,K}(t)$ to decay to zero.

By the energy conservation, every component and its time derivative are bounded
uniformly in $K$. A diagonal Arzel\`a--Ascoli argument therefore gives a
subsequence for which
\[
X_k^K\to X_k \quad \text{locally uniformly on }[0,\infty)
\]
for every fixed $k$, and the limit solves \eqref{Obukhov-inviscid}
componentwise.  For every $\delta>0$, \eqref{H-norm-K} and
\eqref{inviscid-Lyapunov-bound} give the uniform tail estimate
\begin{equation}\label{inviscid-energy-tail}
\sup_K\sup_{t\geq\delta} \sum_{k=L}^{K+1}(X_k^K(t))^2
\lesssim_{q,\delta}N_L^{-2q},
\end{equation}
and, consequently,
\[
\sup_K\sup_{t\geq\delta} X_L^K(t)^2 \lesssim_{q,\delta}N_L^{-2q},
\]
which implies
\[
X^K\to  X \quad\text{in }C([\delta,T];\ell^2)
\]
for every $0<\delta<T$. Hence we can pass to the limit in
\eqref{inviscid-Galerkin-energy} and obtain
$\|X(t)\|_{\ell^2}=1$ for every $t>0$.  More generally, if $s<\alpha$, we can choose
$q$ with $\max\{s,0\}<q<\alpha$, and obtain
$X\in C((0,\infty);H^s)$.

For each fixed $k$, using positivity in the Galerkin approximation, we get
\[
0\leq X_k^K(t) \leq N_k^\alpha\int_0^t(X_{k+1}^K(\tau))^2\,d\tau
\leq N_k^\alpha t.
\]
Passing to the limit, we obtain that $X_k(t)\to0$ as $t\to0+$.  If $s<0$, then
for every $L\geq1$,
\[
\|X(t)\|_{H^s}^2 \leq \sum_{k=1}^{L}N_k^{2s}X_k(t)^2
+N_{L+1}^{2s}\sum_{k>L}X_k(t)^2.
\]
The first term tends to zero for fixed $L$, while the second is at most
$N_{L+1}^{2s}$ by energy conservation.  First sending $t\to0+$ and then
$L\to\infty$ proves the first assertion in
\eqref{inviscid-energy-jump}.

It remains to obtain the rate. By the bound above,
\[
\sum_{k=1}^{L}X_k(t)^2 \leq t^2\sum_{k=1}^{L}N_k^{2\alpha}
\lesssim_{\alpha,\lambda}t^2N_L^{2\alpha}.
\]
Choose $\rho>0$ sufficiently small and, for small $t>0$, select $L=L(t)$ so
that
\[
N_L^\alpha t\leq\rho<N_{L+1}^\alpha t.
\]
Then at least one half of the conserved energy lies in the shells $k>L$.
For $0<r<\alpha$ this gives
\[
\|X(t)\|_{H^r}^2 \geq N_{L+1}^{2r}\sum_{k>L}X_k(t)^2 \geq
\frac12\rho^{2r/\alpha}t^{-2r/\alpha},
\]
which implies \eqref{inviscid-blowup-rate}.

\end{proof}

\begin{proof}[Proof of Corollary~\ref{cor:inviscid-forward-blowup}]
Let $X$ be the solution obtained in Theorem~\ref{th:inviscid}, and
let $Q(X)$ denote the right-hand side of \eqref{Obukhov-inviscid}.  Since
$Q$ is quadratic,
\[
Q(-X)=Q(X).
\]
Define
\[
Y(t)=
\begin{cases}
-X(T-t),&0\leq t<T,\\
0,&t\geq T.
\end{cases}
\]
For $0\leq t<T$,
\[
Y'(t)=X'(T-t)=Q(X(T-t))=Q(-Y(t))=Q(Y(t)),
\]
while the equation is immediate for $t>T$.  For each fixed $k$, the proof of
Theorem~\ref{th:inviscid} gives $X_k\in C^1([0,\infty))$ and $X_k(0)=0$.
Since the three components appearing in $Q_k(X(t))$ tend to zero as
$t\to0+$, we also have $X_k'(t)\to0$.  Thus $Y_k$ and $Y_k'$ glue
continuously to zero at $t=T$, proving the componentwise classical assertion.

The regularity in
\eqref{inviscid-forward-regularity} follows directly from
\eqref{inviscid-energy-jump} and the positive-time regularity of $X$. Finally,
if $0<r<\alpha$, take $\delta_r=\min\{T,t_r\}$.  Whenever
$T-\delta_r<t<T$, we have
\[
\|Y(t)\|_{H^r} =\|X(T-t)\|_{H^r} \geq c_r(T-t)^{-r/\alpha}
\]
by \eqref{inviscid-blowup-rate}. Thus \eqref{inviscid-forward-blowup-rate} holds.

\end{proof}

\textbf{Numerical simulations.} Figures~\ref{fig:type1-type2}, \ref{fig:inviscid}, \ref{fig:DN-comparison}, and \ref{fig:cascade-numerical} depict numerical results produced with assistance from Claude Opus 5.  They are obtained as limits of the Galerkin approximations \eqref{eq:galerkin} of \eqref{Obukhov}, using an implicit Runge--Kutta method of Radau IIA type with analytic Jacobian, at relative tolerance $10^{-10}$ and absolute tolerance $10^{-14}$.  Two regimes need a change of variable or a further refinement.  In the energy-subcritical case the components $X_k$ range over many scales, so we integrate instead the normalized system \eqref{normalized-subcritical} for $Y_k=N_k^{\alpha-2}X_k$, whose components are all $O(1)$; the amplitude $A_K$ of Proposition~\ref{prop-critical-flux} is then found by bisecting $\Phi_K(A)=\max_{t\geq0}Y_1^{K,A}(t)$ to the normalization \eqref{critical-shooting-normalization}.

Figure~\ref{fig:cascade-mixed} is computed from the mixed model \eqref{DNO} and requires a different treatment, because the time scales shrink by the factor $2^{\alpha}R\approx4.1$ from one shell to the next, and 20 such shells are computed.  We therefore integrate the renormalized system obtained from \eqref{DNO} near active shell $n$ by the exact substitution $b_j=X_{n+j}/M$, $\tau=N_n^\alpha Mt$, after which every quantity is $O(1)$. The integration is by an explicit eighth-order Runge--Kutta method at relative tolerance $10^{-11}$.

\bibliography{bib}

\end{document}